\documentclass[hidelinks,onefignum,onetabnum]{siamartdong}

\usepackage{amsfonts}
\usepackage{amssymb}
\usepackage{amsbsy}
\usepackage{latexsym}
\usepackage{bm}
\usepackage{mathrsfs}
\usepackage{graphicx}
\usepackage{epstopdf}
\usepackage{diagbox}
\usepackage{enumerate}
\usepackage{comment}
\usepackage{array}
\usepackage{booktabs}
\usepackage{multirow}
\usepackage{threeparttable}
\usepackage{subfigure}
\usepackage{float}
\usepackage[section]{placeins}
\usepackage{footnote}
\usepackage{arydshln}
\usepackage{algorithmic}

\makesavenoteenv{table}
\numberwithin{equation}{section}
\renewcommand{\arraystretch}{1.24}

\ifpdf
  \DeclareGraphicsExtensions{.pdf,.png,.jpg,.eps}
\else
  \DeclareGraphicsExtensions{.eps}
\fi

\graphicspath{{./ITACPFIGURE/}{./figure/}{./figures/}{./image/}{./images/}{./graphics/}{./graphic/}{./pictures/}{./picture/}}

\newsiamremark{remark}{Remark}
\newsiamremark{assumption}{Assumption}
\newsiamthm{method}{Method}
\newsiamthm{example}{Example}
\newsiamthm{problem}{Problem}

\headers{ITP-C for $\ell_{1-2}$ sparse recovery}{J. Wu, Z. Dong, and J.-F. Yin}

\title{Iterative Thresholding Pursuit with Continuation for
$\ell_{1-2}$-Regularized Sparse Recovery\thanks{\funding{This work was supported by the National Natural Science Foundation of China (Grant No. 12471357).}}}

\author{Junxi Wu\thanks{School of Mathematical Sciences, Tongji University,
No. 1239, Siping Road, Shanghai 200092, China
  (\email{w0521@tongji.edu.cn}).}
\and Zeyu Dong\thanks{Shanghai Research Institute for Intelligent Autonomous Systems, Tongji University,
Shanghai 200092, China
  (\email{dongzeyu@tongji.edu.cn}).}
\and Jun-Feng Yin\thanks{Key Laboratory of Intelligent Computing and Applications (Ministry of Education),
School of Mathematical Sciences, Tongji University,
No. 1239, Siping Road, Shanghai 200092, China
  (\email{yinjf@tongji.edu.cn}).}}

\ifpdf
\hypersetup{
  pdftitle={Iterative Thresholding Pursuit with Continuation for l1-2-Regularized Sparse Recovery},
  pdfauthor={Junxi Wu, Zeyu Dong, and Jun-Feng Yin}
}
\fi

\begin{document}

\maketitle

\begin{abstract}
Sparse recovery aims to reconstruct sparse signals from underdetermined and
possibly noisy linear measurements. Existing $\ell_{1-2}$ iterative
thresholding schemes are first-order methods. We propose an iterative
thresholding pursuit method with continuation (ITP-C) for
$\ell_{1-2}$-regularized sparse recovery. The method goes beyond first-order thresholding by combining the active-set identification capability of the $\ell_{1-2}$ proximal step with a restricted least-squares pursuit step that provides a second-order update on the identified support. The support is generated adaptively by the thresholding
update, and no prior knowledge of the true sparsity level is required. To control the possible instability of the pursuit step while preserving the
descent structure of the continuation scheme, we impose a strict descent check
with respect to the dynamic objective. We establish convergence of the
generated sequence under the Kurdyka--\L{}ojasiewicz framework and prove a local
oracle-type property after correct support identification. Numerical experiments
on synthetic sparse recovery and image reconstruction
illustrate the descent preservation of the proposed safeguard and demonstrate
the improved recovery performance of ITP-C over the state-of-the-art baselines.
\end{abstract}


\section{Introduction}
\label{sec:intro}

Sparse recovery is a key problem in compressed sensing and sparse linear
inverse problems \cite{candes2006robust,donoho2006compressed}. It aims to
reconstruct an unknown compressible signal from underdetermined
noisy observations
\begin{equation}
    \label{eq:linear_model}
    b=Ax+\epsilon,
\end{equation}
where $A\in\mathbb{R}^{m\times n}$ with $m\ll n$ is the sensing matrix,
$b\in\mathbb{R}^m$ is the observation vector, $\epsilon\in\mathbb{R}^m$
denotes the noise, and $x\in\mathbb{R}^n$ is the signal to be recovered.
Such models arise in a broad range of applications, including
medical image reconstruction \cite{lustig2007sparse}, radar signal processing
\cite{potter2010sparsity}, hyperspectral imaging
\cite{bioucas2010alternating}, and high-dimensional statistical learning
\cite{tibshirani1996regression,bickel2009simultaneous}.

A direct formulation of sparse recovery is the $\ell_0$-regularized
least-squares model
\begin{equation}
    \label{eq:l0_model}
    \min_{x\in\mathbb{R}^n}
    \frac12\|Ax-b\|_2^2+\lambda\|x\|_0 ,
\end{equation}
where $\|x\|_0$ counts the number of nonzero entries and $\lambda>0$ is a
regularization parameter. Although equation~\eqref{eq:l0_model} captures
sparsity directly, solving it globally is generally NP-hard
\cite{natarajan1995sparse}. Greedy and hard-thresholding methods, such as
Orthogonal Matching Pursuit (OMP) \cite{tropp2007signal}, Subspace Pursuit
(SP) \cite{dai2009subspace}, CoSaMP \cite{needell2009cosamp}, and Hard
Thresholding Pursuit (HTP) \cite{foucart2011hard}, provide computationally
efficient alternatives to the combinatorial search. These methods can recover
sparse signals under suitable conditions on the sensing matrix, but they
typically require the sparsity level as an input.

A widely used prior-free alternative is the convex $\ell_1$-regularized model
\begin{equation}
    \label{eq:l1_model}
    \min_{x\in\mathbb{R}^n}
    \frac12\|Ax-b\|_2^2+\lambda\|x\|_1 .
\end{equation}
Under suitable conditions such as the restricted isometry property,
$\ell_1$ minimization provides stable recovery guarantees
\cite{candes2008restricted}. It can also be solved efficiently by first-order
and operator-splitting algorithms such as FISTA \cite{beck2009fast} and ADMM
\cite{boyd2011distributed}. However, since the $\ell_1$ penalty penalizes all
nonzero coefficients linearly, it may introduce shrinkage bias on large active
coefficients. This limitation is closely related to the bias issue in
Lasso-type estimators \cite{fan2001variable,candes2008enhancing}. Such bias
can be detrimental in high-accuracy recovery tasks and motivates the use of
sharper continuous nonconvex sparsity-promoting penalties.

To reduce the shrinkage bias while retaining a continuous sparsity-promoting
formulation, various nonconvex penalties have been proposed, including the
SCAD penalty \cite{fan2001variable}, the MCP penalty \cite{zhang2010nearly},
the $\ell_p$ quasi-norm with $0<p<1$ \cite{chartrand2007exact}, and the
$\ell_{1-2}$ penalty \cite{lou2015minimization}. In this paper, we focus on
the $\ell_{1-2}$-regularized model
\begin{equation}
    \label{eq:intro_l12}
    \min_{x\in\mathbb{R}^n}
    F(x):=
    \frac12\|Ax-b\|_2^2+\lambda(\|x\|_1-\|x\|_2).
\end{equation}
The penalty $\|x\|_1-\|x\|_2$ is a difference-of-convex sparsity measure. It
vanishes on one-sparse vectors and penalizes multi-component spread, making it
a sharper continuous surrogate for sparsity than the $\ell_1$ norm in many
settings \cite{yin2015minimization,ge2018null,zhang2021null}. Lou and Yan
derived an explicit form of the $\ell_{1-2}$ proximal mapping and developed
iterative thresholding schemes for this model \cite{lou2018computing}.
Continuation-based ITAC methods further improve the stability of
$\ell_{1-2}$ iterative thresholding and reduce sensitivity to initialization
\cite{hu2025convergence}. Recent convergence theory for continuation-based
$\ell_{1-2}$ thresholding methods has been established in \cite{hu2026l12}.

Despite these developments, pure first-order $\ell_{1-2}$ thresholding may
still make slow local progress after the active set has nearly stabilized. The
proximal thresholding step is effective for identifying a candidate support,
but the active coefficients may still retain regularization-induced shrinkage.
This motivates an active-set refitting strategy, in which the active
coefficients are recomputed by solving a restricted least-squares problem on
the identified support. Related sparsity-constrained pursuit methods and Newton-type variants further show that the active-set refitting strategy and restricted second-order information can improve local behavior under suitable conditions \cite{yuan2014gradient,zhou2021global}.

To this end, we propose an iterative thresholding pursuit method with
continuation, abbreviated as ITP-C, for the $\ell_{1-2}$ sparse recovery model
in equation~\eqref{eq:intro_l12}. The method uses the $\ell_{1-2}$ proximal
step to generate an active set adaptively and then performs a restricted
least-squares pursuit step on the identified support. Since the support is
determined by the thresholding update, ITP-C does not require the true sparsity
level as a prescribed input. A key challenge is that the pursuit step may be
unstable when the identified support is inaccurate, especially in the nonconvex
setting. To address this issue, the pursuit candidate is accepted only when it passes a
strict descent check with respect to the dynamic objective, allowing
active-coefficient refitting to be incorporated while maintaining the descent
property of the continuation scheme. Under the Kurdyka--\L{}ojasiewicz framework, we prove that the generated
sequence converges to a critical point of the target objective. We further
establish local finite support identification and a conditional oracle property:
once the correct support is identified and the pursuit candidate passes the
descent check, ITP-C coincides with the oracle restricted least-squares
estimator and satisfies a stable noise-dependent error bound. Numerical experiments illustrate the role of the pursuit step and
the descent safeguard. In particular, the ablation study shows that an
unchecked pursuit step can produce local increases of the dynamic objective,
whereas the proposed descent check preserves monotonicity. Synthetic sparse
recovery experiments, empirical phase-transition tests, and image reconstruction examples further demonstrate improved recovery
performance over the tested prior-free baselines.

The remainder of this paper is organized as follows. Section
\ref{sec:preliminaries} introduces the notation, variational tools, and the
$\ell_{1-2}$ proximal mapping. Section~\ref{sec:proposed} presents the
proposed ITP-C algorithm and its descent-checked pursuit step. Section
\ref{sec:convergence} establishes the convergence theory, finite support
identification result, and local oracle property. Section
\ref{sec:experiments} reports the numerical experiments. Finally, Section
\ref{sec:conclusion} concludes the paper.

\section{Preliminaries}
\label{sec:preliminaries}

In this section, we first introduce the notation and recall several basic tools
from variational analysis and compressed sensing that will be used throughout
the paper. We then review the baseline ITAC method for the $\ell_{1-2}$ model,
which forms the basis of the proposed ITP-C method.

Throughout this paper, vectors are denoted by lowercase letters (e.g., $x, y$) and matrices by uppercase letters (e.g., $A$). For a vector $x \in \mathbb{R}^n$, its $i$-th component is $x_i$, and its support set is defined as $\operatorname{supp}(x) = \{i : x_i \neq 0\}$. The $\ell_p$-norm ($p \ge 1$) is defined as $\|x\|_p = (\sum_{i=1}^n |x_i|^p)^{1/p}$. Specifically, the $\ell_0$-norm, denoted by $\|x\|_0 = |\operatorname{supp}(x)|$, counts the number of nonzero entries in $x$. 

We formally define the set of all $s$-sparse vectors in $\mathbb{R}^n$ as:
\begin{equation}
    \Sigma_s := \{ x \in \mathbb{R}^n : \|x\|_0 \le s \}.
\end{equation}
For an index subset $T \subseteq \{1, \dots, n\}$, $|T|$ denotes its cardinality. The notation $x_T \in \mathbb{R}^{|T|}$ denotes the subvector restricting $x$ to the indices in $T$, and $A_T \in \mathbb{R}^{m \times |T|}$ represents the submatrix of $A$ consisting of the columns indexed by $T$. The transpose of a matrix is denoted by $A^{\top}$, and $A^\dagger$ represents the Moore-Penrose pseudoinverse, which is given by $(A^{\top} A)^{-1} A^{\top}$ when $A$ has full column rank. The identity matrix is denoted by $I$.

To handle the nonconvex and nonsmooth nature of the sparse penalty considered
in this paper, we recall the limiting subdifferential from variational analysis.

\begin{definition}[\cite{rockafellar1998variational}]
Let $F: \mathbb{R}^n \to \mathbb{R} \cup \{+\infty\}$ be a proper and lower semicontinuous function. The limiting subdifferential (or simply subdifferential) of $F$ at $x \in \operatorname{dom} F$, denoted by $\partial F(x)$, is defined as:
\begin{equation}
    \partial F(x) = \left\{ v \in \mathbb{R}^n : \exists x^k \to x, F(x^k) \to F(x), v^k \in \widehat{\partial} F(x^k) \text{ with } v^k \to v \right\},
\end{equation}
where $\widehat{\partial} F(x)$ is the Fr\'echet subdifferential of $F$ at $x$. A point $x^*$ is called a critical point (or stationary point) of $F$ if it satisfies the generalized Fermat's rule: $0 \in \partial F(x^*)$.
\end{definition}

Another fundamental concept in proximal splitting methods is the proximal
mapping, which dates back to Moreau \cite{moreau1965proximite}.

\begin{definition}
For a proper, lower semicontinuous function $R: \mathbb{R}^n \to \mathbb{R} \cup \{+\infty\}$ and a scalar $\tau > 0$, the proximal mapping of $R$ at a point $y \in \mathbb{R}^n$ is defined as the set of global minimizers:
\begin{equation}
    \operatorname{Prox}_{\tau R}(y) := \arg\min_{x \in \mathbb{R}^n} \left\{ \frac{1}{2}\|x - y\|_2^2 + \tau R(x) \right\}.
\end{equation}
When $R(x)$ is nonconvex, $\operatorname{Prox}_{\tau R}(y)$ may be multi-valued or empty, though it is guaranteed to be nonempty and compact if $R(x)$ is bounded from below.
\end{definition}

For analyzing the global convergence of the full sequence generated by descent algorithms in a nonconvex landscape, the Kurdyka--\L{}ojasiewicz (K\L{}) property is an indispensable geometric tool.

\begin{definition}[\cite{attouch2013convergence, bolte2014proximal}]
\label{def:kl_property}
A proper lower semicontinuous function $F$ is said to have the K\L{} property at $x^* \in \operatorname{dom} \partial F$ if there exists $\eta \in (0, +\infty]$, a neighborhood $U$ of $x^*$, and a continuous concave function $\phi: [0, \eta) \to \mathbb{R}_+$ (with $\phi(0)=0$ and $\phi' > 0$ on $(0, \eta)$), such that for all $x \in U \cap \{x : F(x^*) < F(x) < F(x^*) + \eta\}$, the following K\L{} inequality holds:
\begin{equation}
    \phi'(F(x) - F(x^*)) \operatorname{dist}(0, \partial F(x)) \ge 1.
\end{equation}
\end{definition}
Functions satisfying this property at all points in their domain are known as K\L{} functions. Notably, semi-algebraic functions are universally K\L{} functions \cite{bolte2014proximal}. The next lemma verifies that the $\ell_{1-2}$ objective used in this paper is a
K\L{} function.

\begin{lemma}
\label{lemma:kl_property}
For any fixed $\lambda>0$, define
\[
F_{\lambda}(x)
=
\frac{1}{2}\|Ax-b\|_2^2
+
\lambda(\|x\|_1-\|x\|_2).
\]
Then $F_{\lambda}$ is a proper lower semicontinuous semi-algebraic function.
Consequently, $F_{\lambda}$ satisfies the Kurdyka--\L{}ojasiewicz property at
every point in $\operatorname{dom}\partial F_{\lambda}$.
\end{lemma}

\begin{proof}
The data-fidelity term $\frac12\|Ax-b\|_2^2$ is a quadratic polynomial and is
therefore semi-algebraic. The absolute value function is semi-algebraic since
its graph can be represented by
\[
\{(t,s)\in\mathbb{R}^2:s^2=t^2,\ s\ge 0\}.
\]
Hence $\|x\|_1=\sum_{i=1}^n |x_i|$ is semi-algebraic. Similarly, the Euclidean
norm is semi-algebraic because its graph can be represented by
\[
\left\{(x,s)\in\mathbb{R}^n\times\mathbb{R}:
s^2=\sum_{i=1}^n x_i^2,\ s\ge 0
\right\}.
\]
The class of semi-algebraic functions is closed under finite sums and scalar
multiplication. Therefore,
\[
F_{\lambda}(x)
=
\frac{1}{2}\|Ax-b\|_2^2
+
\lambda(\|x\|_1-\|x\|_2)
\]
is semi-algebraic. Since $F_{\lambda}$ is finite-valued and continuous on
$\mathbb{R}^n$, it is proper and lower semicontinuous. By the standard K\L{} theory for proper lower semicontinuous semi-algebraic
functions \cite{attouch2013convergence,bolte2014proximal}, $F_\lambda$
satisfies the K\L{} property at every point in
$\operatorname{dom}\partial F_\lambda$.
\end{proof}

In the paradigm of compressive sensing, the theoretical guarantee for accurately recovering high-dimensional sparse signals from low-dimensional projections heavily relies on the properties of the sensing matrix $A$.

\begin{definition}[\cite{candes2008restricted}]
\label{def:rip}
A sensing matrix $A \in \mathbb{R}^{m \times n}$ is said to satisfy the Restricted Isometry Property (RIP) of order $s$ if there exists a constant $\delta_s \in (0, 1)$ such that the inequalities
\begin{equation}
    (1 - \delta_s) \|x\|_2^2 \le \|Ax\|_2^2 \le (1 + \delta_s) \|x\|_2^2
\end{equation}
hold simultaneously for all $s$-sparse vectors $x \in \Sigma_s$.
\end{definition}
The RIP fundamentally ensures that every column sub-matrix $A_T$ of size $m \times s$ (where $s \ll m$) behaves nearly like an isometry, securing its full column rank and the well-posedness of the least squares problems restricted to the subspace $T$.

\subsection{The ITAC Algorithm}


Before presenting ITP-C, we review the baseline first-order method ITAC for the $\ell_{1-2}$ model, which provides the foundational proximal thresholding step and continuation strategy used in ITP-C. Given a step size $v \in (0, 1/\|A\|_2^2)$, ITAC updates the intermediate variable $z^k$ through a gradient descent step $$z^{k} = x^k - v A^{\top}(Ax^k - b),$$ followed by the $\ell_{1-2}$ proximal mapping, whose explicit analytical expression is given in Lemma~\ref{lemma:proximal_l12}.

\begin{lemma}[\cite{lou2018computing}]
\label{lemma:proximal_l12}
For any given vector $y \in \mathbb{R}^n$ and parameter $\tau > 0$, the exact solution set of the proximal mapping $x^* \in \operatorname{Prox}_{\tau(\|\cdot\|_1 - \|\cdot\|_2)}(y)$ satisfies:
\begin{itemize}
    \item[(i)] If $\|y\|_\infty > \tau$, then $x^* = \left( 1 + \frac{\tau}{\|z\|_2} \right) z$, where $z := (|y| - \tau)_+ \odot \operatorname{sign}(y)$;
    \item[(ii)] If $\|y\|_\infty = \tau$, then $\|x^*\|_2 = \tau$ and $x_i^* y_i \ge 0$ for all $i \in \{1,\dots,n\}$, with $x_i^* = 0$ when $|y_i| < \tau$;
    \item[(iii)] If $\|y\|_\infty < \tau$, then $x^*$ is a $1$-sparse vector satisfying $\|x^*\|_2 = \|y\|_\infty$ and $x_i^* y_i \ge 0$, with $x_i^* = 0$ when $|y_i| < \|y\|_\infty$.
\end{itemize}
\end{lemma}

In the standard implementation of ITAC, the continuation strategy usually drives
$\lambda_k$ to a relatively small terminal value, so the nondegenerate case
$\|\mathcal{S}_{v\lambda_k}(z^k)\|_2>0$ is typically encountered in practice.
For theoretical completeness, and consistently with the exact proximal
characterization in~\cite{lou2018computing}, we use a deterministic exact selection
from the possibly multi-valued proximal mapping. Setting $\tau=v\lambda_k$, the
baseline ITAC step generates $h^{k+1}$ as follows:
\begin{equation}
\label{eq:itac_update}
h^{k+1} =
\begin{cases}
\left(
1+\dfrac{v\lambda_k}
{\|\mathcal{S}_{v\lambda_k}(z^k)\|_2}
\right)
\mathcal{S}_{v\lambda_k}(z^k),
& \text{if } \|\mathcal{S}_{v\lambda_k}(z^k)\|_2>0, \\[3mm]
z^k_{j_k} e_{j_k},
& \text{if } \|\mathcal{S}_{v\lambda_k}(z^k)\|_2=0,
\end{cases}
\end{equation}
where $\mathcal{S}_{\tau}(\cdot)$ is the standard soft-thresholding operator,
$e_j$ denotes the $j$-th canonical basis vector, and
\begin{equation}
\label{eq:deterministic_index}
j_k := \min \arg\max_{1\le i\le n}|z_i^k|.
\end{equation}
This deterministic tie-breaking rule only affects degenerate cases. 

According to equation~\eqref{eq:itac_update}, the ITAC proximal step first
applies soft-thresholding to remove small components and then rescales the
retained entries. This provides an adaptive mechanism for identifying a
candidate support. Once such a support becomes reliable, the active
coefficients can be further refined by a restricted least-squares step. This
observation motivates the pursuit step introduced in the next section.

\section{The ITP-C algorithm}
\label{sec:proposed}

In this section, we construct the proposed Iterative Thresholding Pursuit
method with Continuation (ITP-C) for the $\ell_{1-2}$ sparse recovery model.
The method builds on the baseline ITAC proximal step reviewed in Section
\ref{sec:preliminaries}, and augments it with a restricted least-squares
pursuit step on the adaptively identified active set. The resulting method goes beyond a purely first-order thresholding update by
combining prior-free support identification with a restricted least-squares
pursuit step, which can be interpreted as a Newton-type update for the
quadratic data-fitting term on the identified support. A strict descent
check is further incorporated to control the possible instability of the
pursuit step and to preserve the descent structure required by the convergence
analysis.

\subsection{Restricted least-squares pursuit step}
\label{subsec:pursuit_step}

Classical pursuit-type sparse recovery algorithms often combine support
selection with a least-squares update on the selected support. For example,
CoSaMP, Subspace Pursuit, and Hard Thresholding Pursuit use thresholding or
support selection to form an active set, followed by a restricted
least-squares step to update the active coefficients
\cite{needell2009cosamp,dai2009subspace,foucart2011hard}. Related
hard-thresholding pursuit methods for sparsity-constrained optimization and
their Newton-type variants further indicate that restricted active-variable
corrections can improve local behavior when the support information is
reliable \cite{yuan2014gradient,zhou2021global}.

This pursuit philosophy is adapted here to the continuous nonconvex
$\ell_{1-2}$ regularized setting. Unlike sparsity-constrained pursuit methods,
the support size is not prescribed in advance. Instead, the active set is
generated by the $\ell_{1-2}$ proximal thresholding step itself. At iteration
$k$, the baseline ITAC step in equation~\eqref{eq:itac_update} generates the
candidate $h^{k+1}$, and we define
\begin{equation}
\label{eq:itpc_active_set}
    T^{k+1}=\operatorname{supp}(h^{k+1}).
\end{equation}
The pursuit step then computes a restricted least-squares estimator on this
identified support:
\begin{equation}
\label{eq:itpc_pursuit_candidate}
    (x_p^{k+1})_{T^{k+1}}=A_{T^{k+1}}^\dagger b,
    \qquad
    (x_p^{k+1})_{(T^{k+1})^c}=0.
\end{equation}
Here the term ``pursuit'' refers to the restricted least-squares refitting on
the support identified by the thresholding step. It is different from classical
sparsity-prescribed pursuit methods, since the support is generated
adaptively and the true sparsity level is not required as an input.

The role of equation~\eqref{eq:itpc_pursuit_candidate} is to recompute the
active coefficients without the regularization penalty. This reduces the
regularization-induced shrinkage left by the thresholding update. The accuracy of this correction therefore depends on the reliability of the
identified active set and on the conditioning of the corresponding sensing
submatrix. The least-squares pursuit step also admits a restricted Newton-type
interpretation for the data-fitting term. For a fixed support $T$, the
restricted quadratic function $u\mapsto \frac12\|A_Tu-b\|_2^2$ has Hessian
$A_T^{\top}A_T$. When $A_T$ has full column rank, the Newton step for this
restricted quadratic reaches the least-squares minimizer
$(A_T^{\top}A_T)^{-1}A_T^{\top}b$ in one step. Thus, once the active set is reliable,
the pursuit step provides a second-order active-set correction to the
first-order thresholding update.


\subsection{Descent safeguard and complete algorithm}
\label{subsec:descent_safeguard}

Although the restricted least-squares pursuit step can substantially improve
the active coefficients, it should not be accepted unconditionally. During
early iterations, the active set identified by the proximal thresholding step
may still contain false indices or miss part of the true support; see, e.g., \cite{belloni2013least,lederer2013trust,chzhen2019lasso}. In that case,
the least-squares candidate may reduce the residual on the selected support
but still increase the nonconvex regularized objective.

Safeguarding aggressive correction steps is a standard idea in Newton-type and
second-order optimization methods, where line-search or descent conditions are
used to prevent unreliable full steps when the local model is inaccurate
\cite{nocedal2006numerical}. In the present nonsmooth nonconvex setting, we
use a direct descent safeguard with respect to the dynamic objective. The
pursuit candidate $x_p^{k+1}$ is accepted only when it strictly improves the dynamic objective relative to the proximal candidate $h^{k+1}$:
\begin{equation}
\label{eq:descent_check}
    x^{k+1} =
    \begin{cases}
    x_p^{k+1}, &
    \text{if } F_k(x_p^{k+1})<F_k(h^{k+1}), \\[1mm]
    h^{k+1}, &
    \text{otherwise}.
    \end{cases}
\end{equation}
Whenever the pursuit candidate is accepted, it gives a strict objective
reduction compared with $h^{k+1}$. If the pursuit candidate does not strictly
improve the dynamic objective, the proximal candidate is retained. In this way,
ITP-C incorporates the active-set least-squares pursuit step while maintaining
the descent property of the continuation scheme.

The proposed ITP-C method is summarized in Algorithm~\ref{alg:itp_c}.

\begin{algorithm}[H]
\caption{The ITP-C algorithm for $\ell_{1-2}$-regularized sparse recovery}
\label{alg:itp_c}
\begin{algorithmic}[1]
\REQUIRE Matrix $A\in\mathbb{R}^{m\times n}$, vector $b\in\mathbb{R}^m$, step size
$v\in(0,1/\|A\|_2^2)$, decay factor $\gamma\in(0,1)$, initial parameter
$\lambda_0$, target parameter $\lambda_{\rm tar}$, maximum iteration number
$K_{\max}$.
\STATE Initialize $x^0=\mathbf{0}$.
\FOR{$k=0,1,\ldots,K_{\max}-1$}
    \STATE Set $\lambda_k=\max\{\lambda_0\gamma^k,\lambda_{\rm tar}\}$.
    \STATE Compute $z^k=x^k-vA^{\top}(Ax^k-b)$, and obtain $h^{k+1}$ by
    equation~\eqref{eq:itac_update} with $\tau=v\lambda_k$.
    \STATE Set $T^{k+1}=\operatorname{supp}(h^{k+1})$.
    \IF{$|T^{k+1}|<m$}
        \STATE Compute $x_p^{k+1}$ by
        $(x_p^{k+1})_{T^{k+1}}=A_{T^{k+1}}^\dagger b$ and
        $(x_p^{k+1})_{(T^{k+1})^c}=0$.
    \ELSE
        \STATE Set $x_p^{k+1}=h^{k+1}$.
    \ENDIF
    \IF{$F_k(x_p^{k+1})<F_k(h^{k+1})$}
        \STATE Set $x^{k+1}=x_p^{k+1}$.
    \ELSE
        \STATE Set $x^{k+1}=h^{k+1}$.
    \ENDIF
    \IF{the stopping criterion is satisfied}
        \STATE Stop and return $\hat{x}=x^{k+1}$.
    \ENDIF
\ENDFOR
\STATE Return $\hat{x}=x^{K_{\max}}$.
\end{algorithmic}
\end{algorithm}

The active set $T^{k+1}$ is generated adaptively by the proximal thresholding
step and is not prescribed by a sparsity level. To avoid an ill-conditioned or
overly dense least-squares correction, the pursuit step is performed only when
$|T^{k+1}|<m$; otherwise, the proximal candidate is retained. The
Moore--Penrose pseudoinverse provides a deterministic minimum-norm solution of
the restricted least-squares problem, which is unique whenever
$A_{T^{k+1}}$ has full column rank. The descent check then ensures that only
objective-decreasing corrections are accepted.

\section{Convergence analysis}
\label{sec:convergence}

In this section, a rigorous convergence theory is established for the proposed ITP-C algorithm. Considering the continuation strategy employed in the actual implementation, the regularization parameter $\lambda_k$ decreases monotonically (i.e., $\lambda_{k+1} \le \lambda_k$). Therefore, the dynamic objective function is defined associated with iteration $k$ as:
\begin{equation}
\label{eq:dynamic_obj}
F_k(x) := f(x) + \lambda_k R(x) = \frac{1}{2}\|Ax - b\|_2^2 + \lambda_k(\|x\|_1 - \|x\|_2),
\end{equation}
where $f(x)=\frac{1}{2}\|Ax-b\|_2^2$ and
$R(x)=\|x\|_1-\|x\|_2$.

\subsection{Coercivity and sufficient decrease}

Before analyzing the sequence generated by the algorithm \ref{alg:itp_c}, we first establish a coercivity
property of the objective function, which will be used to obtain boundedness of
the iterates.

\begin{definition}
\label{def:coercivity}
A continuous function $F: \mathbb{R}^n \to \mathbb{R} \cup \{+\infty\}$ is said to be coercive if $$\lim_{\|x\|_2 \to \infty} F(x)=+\infty.$$
\end{definition}

For the dynamic objective in \eqref{eq:dynamic_obj}, coercivity is not
automatic because the $\ell_{1-2}$ penalty has null directions along one-sparse
vectors. In the compressed sensing setting, this issue can be ruled out by
standard assumptions on the sensing matrix, such as the restricted isometry
property.

\begin{proposition}
\label{prop:coercivity}
For any fixed $\lambda_k>0$, the dynamic objective $F_k$ in equation~\eqref{eq:dynamic_obj} is coercive if and only if the null space of $A$ contains no nonzero one-sparse vector. Equivalently, with $\mathcal{Z}:=\{x\in\mathbb{R}^n\setminus\{0\}:\|x\|_1=\|x\|_2\}$, the necessary and sufficient condition is $\mathcal{N}(A)\cap \mathcal{Z}=\emptyset$.
In particular, if $A$ satisfies the RIP of order one with constant
$\delta_1\in(0,1)$, then $F_k$ is coercive.
\end{proposition}

\begin{proof}

We first prove the sufficiency. Suppose that $\mathcal{N}(A)\cap \mathcal{Z}=\emptyset$. We show that $F_k$ is coercive. Assume, by contradiction, that $F_k$ is not coercive. Then there exists a sequence $\{x^j\}_{j=1}^{\infty}$ such that $\|x^j\|_2\to\infty$ and $F_k(x^j)\le C$ for some constant $C>0$ and all $j$. Let $r_j:=\|x^j\|_2$ and $u^j:=x^j/r_j\in\mathbb{S}^{n-1}$. Since the unit sphere is compact, there exists a subsequence, still denoted by $\{u^j\}$ for simplicity, such that $u^j\to u^*\in\mathbb{S}^{n-1}$.

Using $x^j=r_j u^j$, the boundedness of $F_k(x^j)$ gives
\begin{equation}
\label{eq:coercivity_bound}
\frac{1}{2}\|r_jAu^j-b\|_2^2
+
\lambda_k r_j(\|u^j\|_1-\|u^j\|_2)
\le C .
\end{equation}
Dividing equation~\eqref{eq:coercivity_bound} by $r_j^2$ yields
\[
\frac{1}{2}\|Au^j\|_2^2
-
\frac{\langle Au^j,b\rangle}{r_j}
+
\frac{\|b\|_2^2}{2r_j^2}
+
\frac{\lambda_k(\|u^j\|_1-\|u^j\|_2)}{r_j}
\le
\frac{C}{r_j^2}.
\]
Letting $j\to\infty$ gives $\frac{1}{2}\|Au^*\|_2^2\le0$, hence $Au^*=0$ and $u^*\in\mathcal{N}(A)$.

Next, divide equation~\eqref{eq:coercivity_bound} by $r_j$:
\[
\frac{1}{2}r_j\|Au^j\|_2^2
-
\langle Au^j,b\rangle
+
\frac{\|b\|_2^2}{2r_j}
+
\lambda_k(\|u^j\|_1-\|u^j\|_2)
\le
\frac{C}{r_j}.
\]
Since the first term on the left-hand side is nonnegative, we may discard it
and obtain
\[
-
\langle Au^j,b\rangle
+
\frac{\|b\|_2^2}{2r_j}
+
\lambda_k(\|u^j\|_1-\|u^j\|_2)
\le
\frac{C}{r_j}.
\]
Taking the limit and using $Au^*=0$, we obtain $\lambda_k(\|u^*\|_1-\|u^*\|_2)\le0$. Since $\lambda_k>0$ and $\|u^*\|_1\ge \|u^*\|_2$ for every vector, it follows that $\|u^*\|_1=\|u^*\|_2$. Thus $u^*\in\mathcal{Z}$, and together with $u^*\in\mathcal{N}(A)$ this gives $u^*\in\mathcal{N}(A)\cap\mathcal{Z}$, contradicting $\mathcal{N}(A)\cap\mathcal{Z}=\emptyset$. Hence $F_k$ is coercive.

We next show the converse implication. Suppose that there exists a nonzero vector $u\in\mathcal{N}(A)\cap\mathcal{Z}$. Then $Au=0$ and $\|u\|_1=\|u\|_2$. For any $t>0$, taking $x=tu$ gives $F_k(tu)=\frac12\|b\|_2^2$, while $\|tu\|_2=t\|u\|_2\to\infty$. Therefore, $F_k$ is not coercive, which proves the equivalence.

Finally, if $A$ satisfies the RIP of order one with constant
$\delta_1\in(0,1)$, then every one-sparse vector $u$ satisfies
\[
\|Au\|_2^2\ge (1-\delta_1)\|u\|_2^2>0.
\]
Hence no nonzero one-sparse vector can lie in $\mathcal{N}(A)$, which implies $\mathcal{N}(A)\cap\mathcal{Z}=\emptyset$. Therefore, $F_k$ is coercive.
\end{proof}

By Proposition~\ref{prop:coercivity}, the coercivity of $F_k$, together with its continuity,
implies that each lower level set $\{x\in\mathbb{R}^n:F_k(x)\le \alpha\}$ is
compact. This compactness will be used to establish boundedness of the
generated sequence. We next recall
the descent property of the first-order proximal gradient step. Let
$L=\|A\|_2^2$ be the Lipschitz constant of $\nabla f$.

\begin{lemma}[\cite{lou2018computing, attouch2013convergence}]
\label{lemma:itac_descent}
Suppose the step size satisfies $v \in (0,1/L)$, where
$L=\|A\|_2^2$. For any current iterate $x^k$ and parameter $\lambda_k>0$,
let
\[
z^k=x^k-vA^{\top}(Ax^k-b),
\]
and let $h^{k+1}$ be the deterministic exact proximal selection given
by equation~\eqref{eq:itac_update}. Then the baseline ITAC step satisfies the sufficient
decrease condition
\begin{equation}
\label{eq:itac_descent}
F_k(h^{k+1})
\le
F_k(x^k)
-
c_0\|h^{k+1}-x^k\|_2^2,
\end{equation}
where
\begin{equation}
\label{eq:c0_def}
c_0:=\frac{1-Lv}{2v}>0.
\end{equation}
\end{lemma}

We combine the compactness of the lower level sets with the descent
property of the proximal thresholding step to derive the basic descent and
boundedness properties of ITP-C. When the continuation parameter reaches its terminal value, we write
\(
    F_{\rm tar}(x):=
    \frac12\|Ax-b\|_2^2
    +\lambda_{ tar}(\|x\|_1-\|x\|_2)
\)
for the target objective.

\begin{theorem}
\label{thm:sufficient_decrease}
Assume that $F_{\rm tar}$ is coercive and that the step size satisfies
$v\in(0,1/L)$, where $L=\|A\|_2^2$. Then the sequence $\{x^k\}$ generated by
Algorithm~\ref{alg:itp_c} satisfies:
\begin{itemize}
    \item[(i)] The dynamic objective sequence satisfies $F_{k+1}(x^{k+1})\le F_k(x^k)$ for all $k\ge0$.
    \item[(ii)] There exists a constant $c_0>0$ such that
    \begin{equation}
        \label{eq:thm1_decrease}
        F_{k+1}(x^{k+1})
        \le
        F_k(x^k)
        -
        c_0\|h^{k+1}-x^k\|_2^2,
        \qquad \forall k\ge 0.
    \end{equation}
    \item[(iii)] The proximal displacement vanishes, i.e., $\lim_{k\to\infty}\|h^{k+1}-x^k\|_2=0$.
    \item[(iv)] Both $\{x^k\}$ and $\{h^k\}$ are bounded.
\end{itemize}
\end{theorem}

\begin{proof}
Since $R(x)=\|x\|_1-\|x\|_2\ge0$ and the continuation parameter is nonincreasing, $\lambda_{k+1}\le\lambda_k$, we have
\begin{equation}
\label{eq:proof_lambda_drop}
F_{k+1}(x^{k+1})
=
f(x^{k+1})+\lambda_{k+1}R(x^{k+1})
\le
f(x^{k+1})+\lambda_kR(x^{k+1})
=
F_k(x^{k+1}).
\end{equation}
By the strict descent check in equation~\eqref{eq:descent_check}, either the pursuit
candidate is accepted with a strict improvement over the proximal candidate, or the proximal candidate is retained. Therefore, $F_k(x^{k+1})\le F_k(h^{k+1})$.
Combining this inequality with equation~\eqref{eq:proof_lambda_drop} and Lemma
\ref{lemma:itac_descent} yields
\begin{equation}
\label{eq:thm1_chain}
F_{k+1}(x^{k+1})
\le
F_k(x^{k+1})
\le
F_k(h^{k+1})
\le
F_k(x^k)-c_0\|h^{k+1}-x^k\|_2^2,
\end{equation}
which proves equation~\eqref{eq:thm1_decrease} and the monotonicity property.
Summing equation~\eqref{eq:thm1_decrease} from $k=0$ to $N$ gives
\[
c_0\sum_{k=0}^{N}\|h^{k+1}-x^k\|_2^2
\le
F_0(x^0)-F_{N+1}(x^{N+1}).
\]
Since $F_{\rm tar}$ is coercive and bounded from below, and $F_{N+1}(x^{N+1})\ge F_{\rm tar}(x^{N+1})>\inf_x F_{\rm tar}(x)>-\infty$, the series $\sum_{k=0}^{\infty}\|h^{k+1}-x^k\|_2^2$ is finite. Hence $\|h^{k+1}-x^k\|_2\to0$.

Finally, monotonicity gives $F_{\rm tar}(x^k)\le F_k(x^k)\le F_0(x^0)$. Thus $\{x^k\}$ lies in a bounded lower level set of the coercive function $F_{\rm tar}$, and hence $\{x^k\}$ is bounded. Since $\|h^{k+1}-x^k\|_2\to0$, the sequence $\{h^{k}\}$ is also bounded.
\end{proof}



Theorem~\ref{thm:sufficient_decrease} shows that the descent check preserves
the monotonicity of the dynamic objective. We next show that, after the
continuation parameter reaches its terminal value, the descent-checked pursuit
step can be accepted only finitely many times. Indeed, the restricted
least-squares pursuit step is optimal for the data-fitting term on the selected
support, but it is not necessarily optimal for the full
$\ell_{1-2}$-regularized objective $F_{\rm tar}$. Since there are only
finitely many possible supports, the strict descent check prevents infinitely
many accepted pursuit steps. The next proposition formalizes this
finite-acceptance property.

\begin{proposition}
\label{prop:finite_pursuit_acceptance}
Assume that the strict descent check in equation~\eqref{eq:descent_check} is used and that
the continuation parameter satisfies $\lambda_k=\lambda_{\rm tar}$ for all
sufficiently large $k$. Then the number of accepted pursuit steps is
finite. Consequently, there exists an index $K_{pc}$ such that, for all $k\ge K_{pc}$, the algorithm takes the update $x^{k+1}=h^{k+1}$.
\end{proposition}

\begin{proof}
Since the continuation parameter reaches the terminal value after finitely many
iterations, there exists $K_{\rm tar}>0$ such that $F_k\equiv F_{\rm tar}$ for all $k\ge K_{\rm tar}$. For $k\ge K_{\rm tar}$, the sufficient decrease property of the ITAC candidate gives $F_{\rm tar}(h^{k+1})\le F_{\rm tar}(x^k)$.
If the pursuit candidate is accepted at iteration $k$, then by the strict
descent check,
\[
F_{\rm tar}(x_p^{k+1})
<
F_{\rm tar}(h^{k+1})
\le
F_{\rm tar}(x^k).
\]
Thus every accepted pursuit step produces a strict decrease of the target
objective value.

Next, observe that every pursuit candidate is obtained by solving a restricted
least-squares problem on a support set selected from
$\{1,\ldots,n\}$. Since there are only finitely many support sets, all possible
deterministic Moore--Penrose least-squares candidates belong to the finite set
\[
\mathcal{V}_{LS}
:=
\left\{
x\in\mathbb{R}^n:
x_T=A_T^\dagger b,\;
x_{T^c}=0,\;
T\subseteq\{1,\ldots,n\}
\right\}.
\]
Therefore, the set of possible objective values generated by accepted pursuit
candidates,
\[
F_{\rm tar}(\mathcal{V}_{LS})
:=
\{F_{\rm tar}(x):x\in\mathcal{V}_{LS}\},
\]
is finite.
The same objective value in $F_{\rm tar}(\mathcal{V}_{LS})$ cannot be accepted twice after $K_{\rm tar}$. Indeed, suppose that a pursuit step is accepted at iteration $k_1\ge K_{\rm tar}$ and produces $F_{\rm tar}(x^{k_1+1})=\alpha$. By monotonicity of the target objective on the tail sequence, $F_{\rm tar}(x^k)\le\alpha$ for any later iteration $k\ge k_1+1$. If a later pursuit step at some $k_2\ge k_1+1$ were accepted with the same objective value $\alpha$, then
\[
\alpha
=
F_{\rm tar}(x_p^{k_2+1})
<
F_{\rm tar}(h^{k_2+1})
\le
F_{\rm tar}(x^{k_2})
\le
\alpha,
\]
which is impossible. Hence each value in the finite set
$F_{\rm tar}(\mathcal{V}_{LS})$ can be accepted at most once.

Therefore, only finitely many pursuit steps can be accepted after
$K_{\rm tar}$. Since there are also only finitely many iterations before
$K_{\rm tar}$, the total number of accepted pursuit steps is finite. Hence
there exists $K_{p}$ such that for all $k\ge K_{p}$, the pursuit candidate is not accepted and the algorithm takes the update $x^{k+1}=h^{k+1}$. This completes the proof.
\end{proof}
\subsection{Subsequence convergence to stationarity}


Having established sufficient decrease and boundedness, we next study the
stationarity of accumulation points. Since the $\ell_{1-2}$ objective is
nonsmooth and nonconvex, criticality is characterized by the limiting
subdifferential.

\begin{lemma}
\label{lemma:relative_error}
Let $L=\|A\|_2^2$ be the Lipschitz constant of $\nabla f$. For any iteration
$k$, let $h^{k+1}$ be the deterministic exact proximal selection
generated by equation~\eqref{eq:itac_update}. Then there exists $\xi^{k+1}\in \partial F_k(h^{k+1})$ such that
\begin{equation}
\label{eq:relative_error_bound}
\|\xi^{k+1}\|_2
\le
\left(L+\frac{1}{v}\right)
\|h^{k+1}-x^k\|_2 .
\end{equation}
In particular, for all $k\ge K_{tar}$, where $\lambda_k=\lambda_{\rm tar}$, we have $\xi^{k+1}\in \partial F_{\rm tar}(h^{k+1})$.
\end{lemma}

\begin{proof}
By the proximal optimality condition of the ITAC step,
\[
h^{k+1}
\in
\operatorname{Prox}_{v\lambda_k R}
\left(x^k-v\nabla f(x^k)\right),
\]
we have
\[
0
\in
h^{k+1}-x^k+v\nabla f(x^k)
+
v\lambda_k\partial R(h^{k+1}).
\]
Equivalently,
\[
\frac{1}{v}(x^k-h^{k+1})-\nabla f(x^k)
\in
\lambda_k\partial R(h^{k+1}).
\]
Define
\[
\xi^{k+1}
:=
\nabla f(h^{k+1})
-
\nabla f(x^k)
+
\frac{1}{v}(x^k-h^{k+1}).
\]
Using the subdifferential sum rule $\partial F_k(x)=\nabla f(x)+\lambda_k\partial R(x)$, we obtain $\xi^{k+1}\in \partial F_k(h^{k+1})$.
Moreover, by the $L$-Lipschitz continuity of $\nabla f$,
\[
\begin{aligned}
\|\xi^{k+1}\|_2
&\le
\|\nabla f(h^{k+1})-\nabla f(x^k)\|_2
+
\frac{1}{v}\|x^k-h^{k+1}\|_2  \\
&\le
\left(L+\frac{1}{v}\right)
\|h^{k+1}-x^k\|_2 .
\end{aligned}
\]
This proves equation~\eqref{eq:relative_error_bound}. When $k\ge K_{\rm tar}$, we have $\lambda_k=\lambda_{\rm tar}$ and hence $F_k=F_{\rm tar}$, which gives $\xi^{k+1}\in \partial F_{\rm tar}(h^{k+1})$.
\end{proof}

Lemma~\ref{lemma:relative_error} shows that the stationarity residual of the
proximal candidate is controlled by the proximal displacement. Together with
Theorem~\ref{thm:sufficient_decrease}, this implies that the residual vanishes
along any convergent subsequence. We now use this fact to identify the limit
points of the generated sequence.

\begin{theorem}
\label{thm:subsequence_convergence}
Suppose the conditions in Theorem~\ref{thm:sufficient_decrease} hold. Let
$\{x^k\}$ be the sequence generated by Algorithm~\ref{alg:itp_c}. Then the
sequence is bounded, and every accumulation point $x^*$ of $\{x^k\}$ is a
critical point of the target objective $F_{\rm tar}$, namely $0\in \partial F_{\rm tar}(x^*)$.
\end{theorem}

\begin{proof}
By Theorem~\ref{thm:sufficient_decrease}, the sequence $\{x^k\}$ is bounded.
Hence, by the Bolzano--Weierstrass theorem, there exists a convergent
subsequence $\{x^{k_j}\}$ such that $x^{k_j}\to x^*$. Theorem~\ref{thm:sufficient_decrease}(iii) gives $\|h^{k+1}-x^k\|_2\to0$, and therefore $h^{k_j+1}\to x^*$.

For sufficiently large $j$, we have $k_j\ge K_{tar}$ and hence
$F_{k_j}\equiv F_{\rm tar}$. By Lemma \ref{lemma:relative_error}, there exists $\xi^{k_j+1}\in \partial F_{\rm tar}(h^{k_j+1})$ such that $\|\xi^{k_j+1}\|_2\le (L+1/v)\|h^{k_j+1}-x^{k_j}\|_2$. Thus $\xi^{k_j+1}\to0$.

Since $F_{\rm tar}$ is continuous, the convergence $h^{k_j+1}\to x^*$ implies $F_{\rm tar}(h^{k_j+1})\to F_{\rm tar}(x^*)$. Using the closedness property of the limiting subdifferential under function-value convergence \cite[Theorem 8.6]{rockafellar1998variational}, we obtain from $(h^{k_j+1},\xi^{k_j+1})\in\operatorname{gph}(\partial F_{\rm tar})$ that $0\in \partial F_{\rm tar}(x^*)$.
Therefore, every accumulation point of $\{x^k\}$ is a critical point of
$F_{\rm tar}$.
\end{proof}

\subsection{Global convergence via the K\L{} property}

The preceding subsection establishes subsequential stationarity: every
accumulation point of the generated sequence is a critical point of
$F_{\rm tar}$. It remains to prove full-sequence convergence. To this end, we
combine the sufficient decrease property, the relative-error bound, the finite
acceptance of pursuit corrections, and the K\L{} property of $F_{\rm tar}$, which
follows from Lemma~\ref{lemma:kl_property}.

\begin{theorem}
\label{thm:global_convergence}
Suppose the conditions in Theorem~\ref{thm:sufficient_decrease} hold and the
strict descent check in equation~\eqref{eq:descent_check} is used. Then the sequence
$\{x^k\}$ generated by Algorithm~\ref{alg:itp_c} has the finite-length
property:
\[
\sum_{k=0}^{\infty}\|x^{k+1}-x^k\|_2<\infty.
\]
Consequently, $\{x^k\}$ is a Cauchy sequence and converges to a critical point
$x^*$ of the target objective
\[
F_{\rm tar}(x)=f(x)+\lambda_{\rm tar}R(x).
\]
\end{theorem}

\begin{proof}
By Proposition~\ref{prop:finite_pursuit_acceptance}, there exists \(K_{p}\) such that, for all \(k\ge K_{p}\), the pursuit candidate is not accepted and hence \(x^{k+1}=h^{k+1}\). Since the continuation parameter reaches its terminal value after finitely many iterations, by increasing \(K_{p}\) if necessary, we may assume that \(\lambda_k=\lambda_{\rm tar}\) and \(F_k\equiv F_{\rm tar}\) for all \(k\ge K_{p}\). For simplicity, write \(F:=F_{\rm tar}\) and \(d_k:=\|x^{k+1}-x^k\|_2\). Then, for all \(k\ge K_{p}\), the sufficient decrease inequality gives
\[
F(x^{k+1})\le F(x^k)-c_0 d_k^2 ,
\]
where \(c_0>0\). Moreover, by Lemma~\ref{lemma:relative_error}, there exists \(\xi^{k+1}\in \partial F(x^{k+1})\) such that \(\|\xi^{k+1}\|_2\le C d_k\), where \(C:=L+1/v\).

The sequence \(\{F(x^k)\}_{k\ge K_{p}}\) is nonincreasing. Since \(F\) is bounded from below on the bounded sequence \(\{x^k\}\), there exists \(\zeta\in\mathbb R\) such that \(F(x^k)\to \zeta\). Also, summing the sufficient decrease inequality gives \(\sum_{k=K_{p}}^\infty d_k^2<\infty\), and hence \(d_k\to0\).

Let \(\Omega\) be the set of accumulation points of \(\{x^k\}\). Since \(\{x^k\}\) is bounded, \(\Omega\) is nonempty and compact. By the continuity of \(F\) and the convergence \(F(x^k)\to\zeta\), we have \(F\equiv \zeta\) on \(\Omega\). Since \(F\) is a K\L{} function, the uniformized K\L{} property on the compact set \(\Omega\) implies that there exist \(\eta>0\), a neighborhood \(U\) of \(\Omega\), and a concave function \(\phi\) with \(\phi(0)=0\) and \(\phi'>0\), such that
\[
\phi'(F(x)-\zeta)\,{\rm dist}(0,\partial F(x))\ge 1
\]
whenever \(x\in U\) and \(\zeta<F(x)<\zeta+\eta\). Since \(F(x^k)\to \zeta\) and \({\rm dist}(x^k,\Omega)\to0\), there exists \(K\ge K_{p}+1\) such that the above K\L{} inequality holds at \(x^k\) for all \(k\ge K\), unless \(F(x^k)=\zeta\).

If \(F(x^{k_0})=\zeta\) for some \(k_0\ge K\), then the monotonicity of \(F(x^k)\) gives \(F(x^k)=\zeta\) for all \(k\ge k_0\). The sufficient decrease inequality then yields \(d_k=0\) for all \(k\ge k_0\), and the finite-length property follows immediately. Hence we only need to consider the case \(h_k:=F(x^k)-\zeta>0\) for all sufficiently large \(k\).

For \(k\ge K\), the relative-error bound at \(x^k\) gives \({\rm dist}(0,\partial F(x^k))\le C d_{k-1}\). Therefore, the K\L{} inequality implies \(\phi'(h_k)\ge 1/(C d_{k-1})\). By the concavity of \(\phi\), we have \(\phi(h_k)-\phi(h_{k+1})\ge \phi'(h_k)(h_k-h_{k+1})\). Combining this with \(h_k-h_{k+1}=F(x^k)-F(x^{k+1})\ge c_0 d_k^2\), we obtain
\[
\phi(h_k)-\phi(h_{k+1})\ge \frac{c_0}{C}\frac{d_k^2}{d_{k-1}} .
\]
Let \(\Delta_k:=\phi(h_k)-\phi(h_{k+1})\). Then \(d_k^2\le (C/c_0)\Delta_k d_{k-1}\). Using \(2\sqrt{ab}\le a+b\), we get
\[
d_k\le \frac12 d_{k-1}+\frac{C}{2c_0}\Delta_k .
\]
Summing this inequality from \(k=K\) to \(N\) gives
\[
\sum_{k=K}^N d_k
\le \frac12\sum_{k=K}^N d_{k-1}
+\frac{C}{2c_0}\sum_{k=K}^N\Delta_k .
\]
Since \(\sum_{k=K}^N d_{k-1}\le d_{K-1}+\sum_{k=K}^N d_k\) and \(\sum_{k=K}^N\Delta_k=\phi(h_K)-\phi(h_{N+1})\le \phi(h_K)\), it follows that
\[
\sum_{k=K}^N d_k
\le d_{K-1}+\frac{C}{c_0}\phi(h_K).
\]
The right-hand side is independent of \(N\). Letting \(N\to\infty\), we obtain \(\sum_{k=K}^\infty d_k<\infty\). Adding the finite number of preceding terms yields
\[
\sum_{k=0}^\infty \|x^{k+1}-x^k\|_2<\infty .
\]
\end{proof}

\subsection{Finite support identification and local oracle property}

The convergence results above guarantee that the generated sequence approaches
a critical point of the target objective. We now turn to the local behavior of
the method near the true sparse signal, where the role of the pursuit step can
be made more explicit. The key question is whether the proximal thresholding
step can identify the true support once the iterate is sufficiently close to
the signal. If the active and inactive coordinates are separated by the
terminal threshold, the answer is positive: the $\ell_{1-2}$ proximal step
recovers the correct support. On this support, the restricted least-squares
pursuit step coincides with the oracle estimator, provided that it passes the
strict descent check.

For the local analysis, assume that the noisy observation is generated by
$b=Ax^*+\epsilon$, where $x^*$ is an $s$-sparse signal with support
$T^*=\operatorname{supp}(x^*)$. Let
$x_{\min}:=\min_{i\in T^*}|(x^*)_i|$ be the minimum nonzero magnitude of
the true signal.

\begin{theorem}
\label{thm:support_identification}
Suppose that
$\lambda_k=\lambda_{\rm tar}$ at iteration $k$. Define
$$M_{2,\infty}:=\|I-vA^{\top}A\|_{2\to\infty}
=\max_{1\le i\le n}\|e_i^{\top}(I-vA^{\top}A)\|_2.$$ If the current iterate satisfies
$\|x^k-x^*\|_2\le \rho$ and the parameters satisfy
\begin{equation}
\label{eq:rho_condition}
M_{2,\infty}\rho+v\|A^{\top}\epsilon\|_{\infty}
<
v\lambda_{\rm tar}
<
x_{\min}-M_{2,\infty}\rho-v\|A^{\top}\epsilon\|_{\infty},
\end{equation}
then the proximal thresholding step identifies the true support:
$\operatorname{supp}(h^{k+1})=T^*$.
\end{theorem}

\begin{proof}
Recall that $z^k=x^k-vA^{\top}(Ax^k-b)$. Since
$b=Ax^*+\epsilon$, we have
\[
z^k
=
x^*
+
(I-vA^{\top}A)(x^k-x^*)
+
vA^{\top}\epsilon .
\]
For any $i\in T^*$, the reverse triangle inequality gives
\[
|z_i^k|
\ge
x_{\min}-M_{2,\infty}\rho-v\|A^{\top}\epsilon\|_\infty .
\]
By the right-hand inequality in equation~\eqref{eq:rho_condition}, it follows
that $|z_i^k|>v\lambda_{\rm tar}$ for all $i\in T^*$. Hence
$\mathcal{S}_{v\lambda_{\rm tar}}(z^k)\neq 0$, and the nondegenerate branch of
the proximal update in equation~\eqref{eq:itac_update} applies. In this case, $h^{k+1}$ is a positive scalar multiple of
$\mathcal{S}_{v\lambda_{\rm tar}}(z^k)$, and all coordinates in $T^*$ remain
nonzero in $h^{k+1}$.

For any $i\notin T^*$, we have $(x^*)_i=0$, and the same expansion
yields
\[
|z_i^k|
\le
M_{2,\infty}\rho+v\|A^{\top}\epsilon\|_\infty .
\]
By the left-hand inequality in equation~\eqref{eq:rho_condition}, we obtain
$|z_i^k|<v\lambda_{\rm tar}$ for all $i\notin T^*$. Thus the corresponding
soft-thresholded coordinates vanish, and so do the corresponding coordinates
of $h^{k+1}$. Therefore
\[
\operatorname{supp}(h^{k+1})=T^* .
\]
\end{proof}

\begin{remark}
\label{remark:noise_separation}
The separation condition in equation~\eqref{eq:rho_condition} contains both a lower and an
upper restriction on the terminal threshold $v\lambda_{\rm tar}$:
\[
M_{2,\infty}\rho+v\|A^{\top}\epsilon\|_{\infty}
<
v\lambda_{\rm tar}
<
x_{\min}-M_{2,\infty}\rho-v\|A^{\top}\epsilon\|_{\infty}.
\]
The lower bound prevents inactive coordinates from surviving the
soft-thresholding step, whereas the upper bound prevents true active
coordinates from being removed. Hence the condition requires a nonempty
separation interval, equivalently
\[
x_{\min}
>
2\left(M_{2,\infty}\rho+v\|A^{\top}\epsilon\|_{\infty}\right).
\]
This is a standard signal-to-noise type requirement for support recovery \cite{wainwright2009sharp}, requiring the minimum
nonzero coefficient to dominate the perturbation caused by the current
iteration error and the noise back-projection.

The term $\|A^{\top}\epsilon\|_\infty$ measures the maximum correlation between the
noise vector and the columns of the sensing matrix. If
$\epsilon\sim\mathcal{N}(0,\sigma^2 I_m)$ and $a_j$ denotes the $j$-th column
of $A$, then
\[
a_j^{\top}\epsilon\sim \mathcal{N}(0,\sigma^2\|a_j\|_2^2).
\]
Consequently, a standard Gaussian maximum bound gives, with high probability
over the noise,
\[
\|A^{\top}\epsilon\|_\infty
\lesssim
\sigma \left(\max_{1\le j\le n}\|a_j\|_2\right)\sqrt{2\log n}.
\]
In particular, when the sensing matrix is column-normalized so that
$\|a_j\|_2\le 1$, this term is of order
$\mathcal{O}(\sigma\sqrt{\log n})$. Thus the admissible range of
$\lambda_{\rm tar}$ is governed by the noise level, the conditioning of the
sensing matrix through $M_{2,\infty}$, the current distance $\rho$ to the true
signal, and the minimum signal amplitude $x_{\min}$.
\end{remark}

Theorem~\ref{thm:support_identification} gives a local condition under which
the proximal thresholding step identifies the true support. Once this occurs,
the pursuit step is performed on the correct active set and its effect can be
described explicitly. In the noiseless case, the restricted least-squares
candidate recovers the true signal in one step, provided that it is accepted by
the strict descent check. In the noisy case, the same candidate coincides with
the oracle restricted least-squares estimator and satisfies a stable
noise-dependent error bound. The next theorem formalizes these properties.

\begin{theorem}
\label{thm:local_oracle_property}
Assume that the conditions of Theorem~\ref{thm:support_identification} hold at
iteration $k$, so that the proximal thresholding step identifies the true support
$\operatorname{supp}(h^{k+1})=T^*$. Suppose further that $s=|T^*|<m$ and
$A_{T^*}$ has full column rank. Then the pursuit candidate generated by
Algorithm~\ref{alg:itp_c} is the oracle least-squares estimator on the true
support:
\begin{equation}
\label{eq:local_oracle_ls}
(x_p^{k+1})_{T^*}
=
(A_{T^*}^{\top}A_{T^*})^{-1}A_{T^*}^{\top}b,
\qquad
(x_p^{k+1})_{(T^*)^c}=0.
\end{equation}
In particular, if $b=Ax^*$, then $x_p^{k+1}=x^*$. If the
oracle pursuit candidate satisfies the strict descent check,
then it is accepted by Algorithm~\ref{alg:itp_c}, namely
$x^{k+1}=x_p^{k+1}$.

Moreover, suppose that $b=Ax^*+\epsilon$ with
$\|\epsilon\|_2\le \varepsilon$. If $A$ satisfies the RIP of order $s$ with
constant $\delta_s\in(0,1)$ and the oracle pursuit candidate is accepted, then
the accepted estimate satisfies
\begin{equation}
\label{eq:local_oracle_error_bound}
\|x^{k+1}-x^*\|_2
\le
\frac{\varepsilon}{\sqrt{1-\delta_s}}.
\end{equation}
\end{theorem}

\begin{proof}
By Theorem~\ref{thm:support_identification}, the proximal thresholding step identifies
$T^*$. Hence the pursuit stage solves the restricted least-squares problem on
$T^*$. Since $A_{T^*}$ has full column rank, this problem has the unique
solution given in equation~\eqref{eq:local_oracle_ls}.

In the noiseless case, $b=Ax^*=A_{T^*}(x^*)_{T^*}$. Therefore
equation~\eqref{eq:local_oracle_ls} gives
\[
(x_p^{k+1})_{T^*}=(x^*)_{T^*},
\qquad
(x_p^{k+1})_{(T^*)^c}=0,
\]
and hence $x_p^{k+1}=x^*$. If the strict descent check holds, the
acceptance statement follows directly from Algorithm~\ref{alg:itp_c}.

It remains to prove the noisy oracle error bound. Since the oracle pursuit
candidate is accepted, $x^{k+1}=x_p^{k+1}$. Using
$b=A_{T^*}(x^*)_{T^*}+\epsilon$, equation~\eqref{eq:local_oracle_ls}
yields
\[
(x^{k+1})_{T^*}-(x^*)_{T^*}
=
A_{T^*}^{\dagger}\epsilon,
\qquad
(x^{k+1})_{(T^*)^c}=(x^*)_{(T^*)^c}=0.
\]
Consequently,
\[
\|x^{k+1}-x^*\|_2
\le
\|A_{T^*}^{\dagger}\|_2\|\epsilon\|_2.
\]
By the RIP of order $s$, $\sigma_{\min}(A_{T^*})\ge\sqrt{1-\delta_s}$, and hence
$\|A_{T^*}^{\dagger}\|_2\le 1/\sqrt{1-\delta_s}$. Combining this estimate with
$\|\epsilon\|_2\le\varepsilon$ gives equation~\eqref{eq:local_oracle_error_bound}.
\end{proof}

Theorem~\ref{thm:local_oracle_property} explains how the pursuit step removes
the regularization-induced shrinkage after correct support identification.
Continuous regularized estimators, such as Lasso-type methods, may exhibit
shrinkage on the active coefficients, and their statistical error bounds often
contain terms depending explicitly on the regularization parameter, for
instance of order $\mathcal{O}(\lambda\sqrt{s})$ under standard sparse
recovery conditions \cite{bickel2009simultaneous, fan2001variable,
zou2006adaptive}. Similar regularization-dependent error terms also appear in
first-order $\ell_{1-2}$ thresholding analyses; see, for example, Remark~2.1
in \cite{hu2026l12}.

In contrast, Theorem~\ref{thm:local_oracle_property} shows that once the
correct support has been identified by the $\ell_{1-2}$ proximal step and the
corresponding pursuit candidate is accepted by the strict descent check,
ITP-C replaces the regularized proximal iterate with the restricted
least-squares estimator on the identified active set. In this conditional
oracle regime, the shrinkage bias on the active coefficients is removed, and
the resulting error bound depends only on the noise level and the conditioning
of $A_{T^*}$.


\section{Numerical experiments}
\label{sec:experiments}

In this section, we evaluate the proposed ITP-C method through synthetic sparse
recovery and image reconstruction.
All simulations are implemented in MATLAB R2024b and run on a desktop computer
with an Intel Core Ultra 9 185H CPU and 32 GB of RAM.

The compared methods include representative prior-free sparse recovery
solvers: the convex baseline L1-FISTA \cite{beck2009fast}, the reweighted
$\ell_1$ method \cite{candes2008enhancing}, L1-2 PLDCA
\cite{yin2015minimization}, L1-2 PGD \cite{lou2018computing}, and the baseline
L1-2 ITAC \cite{hu2026l12}. 

For the synthetic experiments, we test two sensing structures:
\begin{itemize}
    \item \textbf{Gaussian matrix.}
    The Gaussian sensing matrix is generated by
    \[
    A_{i,j}\overset{i.i.d.}{\sim}\mathcal{N}(0,1/m),
    \qquad
    i=1,\ldots,m,\quad j=1,\ldots,n.
    \]
    This matrix serves as a standard compressed sensing benchmark.

    \item \textbf{PDCT-type matrix.}
    To examine performance beyond the Gaussian setting, we also use a
    PDCT-type random cosine sensing matrix defined by
    \[
    A_{i,j}
    =
    \frac{1}{\sqrt{m}}\cos(2\pi j\xi_i),
    \qquad
    i=1,\ldots,m,\quad j=1,\ldots,n,
    \]
    where $\xi=(\xi_1,\ldots,\xi_m)^{\top}$ and
    $\xi_i\overset{i.i.d.}{\sim}\mathcal{U}[0,1]$.
\end{itemize}
Random Gaussian and PDCT-type sensing matrices are standard choices in
compressed sensing, since they are incoherent and have small restricted
isometry constants with high probability under suitable sampling conditions
\cite{candes2006robust,blanchard2011compressed}.

In the synthetic experiments, the observation vector is generated by
\[
    b = A x^\ast + \sigma \epsilon,
\]
where \(x^\ast\) is the ground-truth sparse signal and 
\(\epsilon \sim \mathcal{N}(0,I_m)\). The parameter \(\sigma\) 
controls the noise level.

The experiments are organized as follows. Section~5.1 examines the role of the
descent safeguard by comparing the dynamic objective trajectories of ITAC,
ITP-C without the descent check, and the proposed ITP-C. Section~5.2 compares
ITP-C with representative prior-free solvers in terms of time-to-accuracy,
relative recovery error, and empirical success rates under Gaussian and
PDCT-type sensing matrices. Multiple noise levels
$\sigma\in\{0,0.001,0.0015,0.002\}$ are used in the accuracy-oriented
computational study, while representative noiseless and noisy cases are shown
for the trajectory and phase-transition experiments. Section~5.3 evaluates the
method on image reconstruction under Gaussian measurement noise with standard
deviation $\sigma=10^{-3}$.

\subsection{Descent safeguard and dynamic objective behavior}
\label{subsec:ablation}

The first experiment examines the role of the strict descent check in ITP-C.
To isolate the effect of this safeguard, we also include an unchecked variant,
denoted by ITP-C (no check), in which the pursuit candidate is accepted
whenever it is computed. 


The sensing matrix $A \in \mathbb{R}^{m \times n}$ is generated with
$m = 512$ and $n = 1024$, and the true signal $x^*$ has sparsity
$s = 50$. We compare the dynamic objective trajectories over 100 iterations
under both noiseless ($\sigma = 0$) and noisy ($\sigma = 0.001$) settings.

\begin{figure}[!t]
    \centering
    \subfigure[Gaussian sensing matrix]{
        \includegraphics[width=0.86\textwidth]{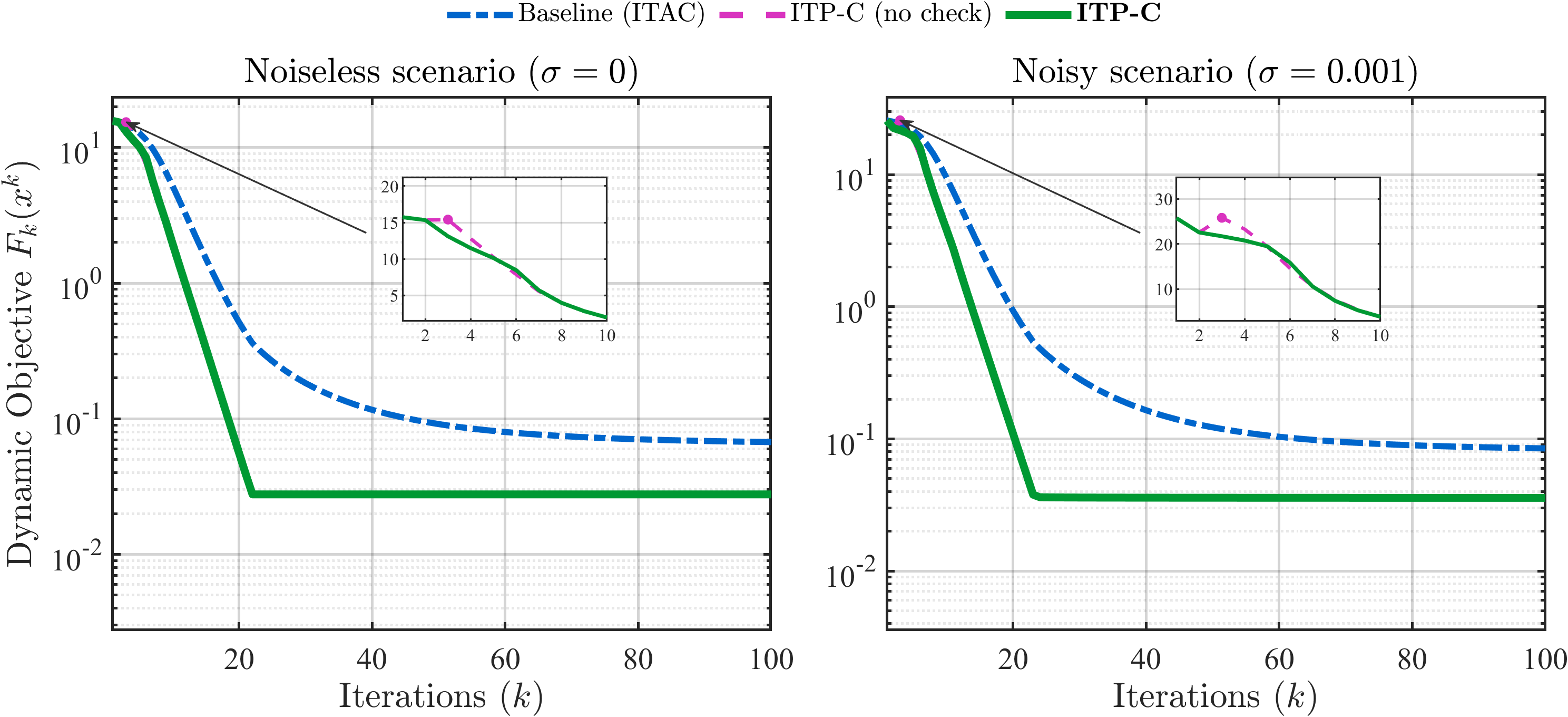}
    }\\[-0.5mm]
    \subfigure[PDCT-type sensing matrix]{
        \includegraphics[width=0.86\textwidth]{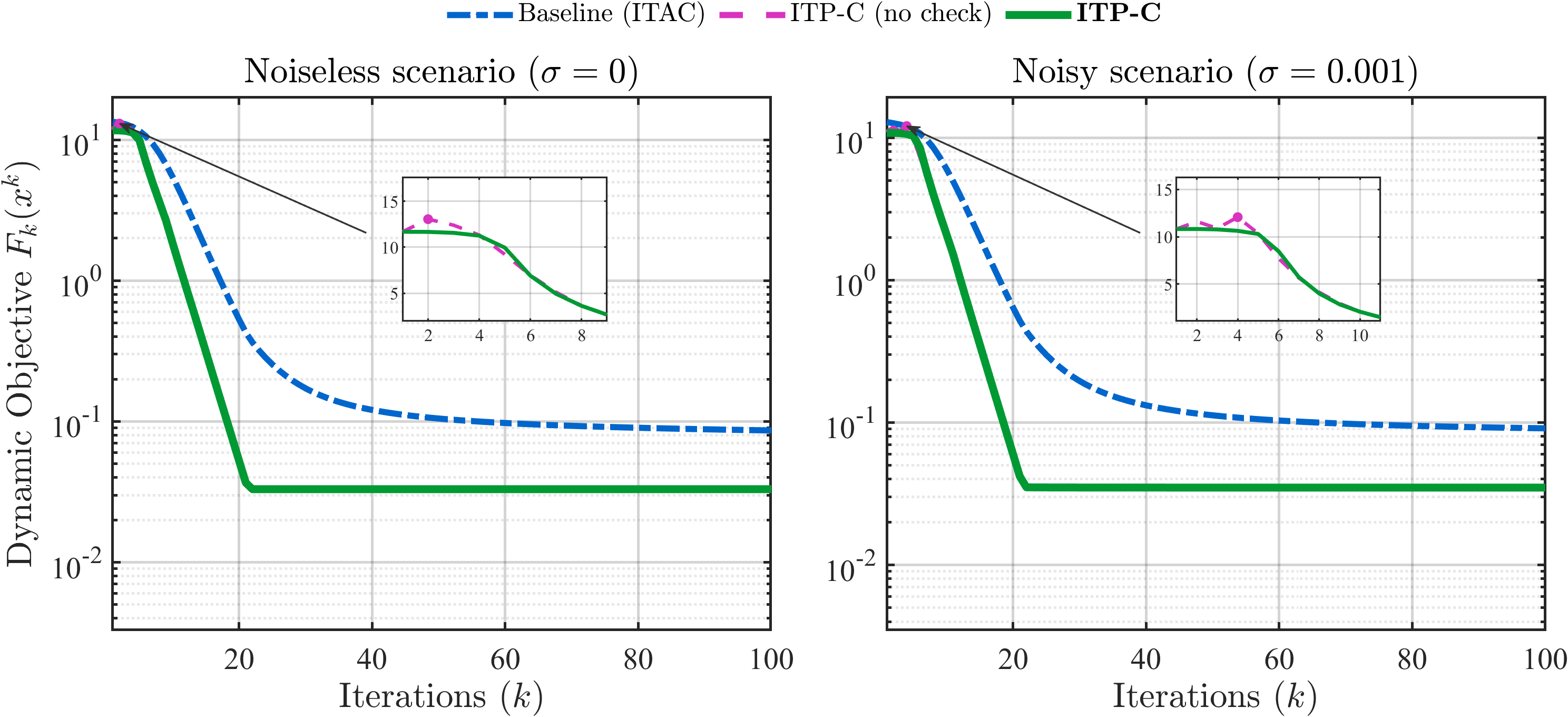}
    }
    \vspace{-1mm}
    \caption{Evolution of the dynamic objective value $F_k(x^k)$.
    For each sensing matrix, the left panel corresponds to the noiseless case
    ($\sigma=0$) and the right panel corresponds to the noisy case
    ($\sigma=0.001$). The magnified insets highlight representative local
    increases of the unchecked variant.}
    \label{fig:convergence}
\end{figure}

Figure~1 shows that the descent-checked ITP-C decreases the dynamic objective
more rapidly than the baseline ITAC method. This behavior is consistent with
the active-set least-squares correction, which exploits the local quadratic
structure of the data-fitting term on the identified support. The unchecked
variant illustrates the role of the safeguard: without testing the dynamic
objective, the least-squares pursuit step may produce local upward jumps in
$F_k(x^k)$, as shown in the magnified insets. Such jumps can occur when the
identified support is inaccurate, since reducing the residual on the selected
support does not necessarily decrease the full nonconvex regularized objective.

\subsection{Comparison with representative sparse recovery algorithms}
\label{subsec:comparison}

We next compare ITP-C with representative prior-free sparse recovery solvers.
The signal dimension
is set to $n=1024$, the number of measurements is $m=512$, and the sparsity
level is $s=50$. We consider the noiseless case as a reference and three noisy
settings with $\sigma\in\{0.001,0.0015,0.002\}$. All algorithms are run with
a maximum iteration limit of $K_{\max}=800$, and the practical stopping
criterion is based on the relative step difference
\[
\frac{\|x^k-x^{k-1}\|_2}{\max(1,\|x^{k-1}\|_2)}<10^{-5}.
\]
The relative step-difference criterion is used as the practical stopping rule.
For these synthetic experiments, where the ground truth is available, we also
record the first CPU time and iteration number at which the relative recovery
error ${\rm RelErr}(x^k):=\|x^k-x^{*}\|_2/\|x^{*}\|_2$ falls
below $10^{-2}$. The final relative error after termination is also
reported to indicate the resulting recovery accuracy.

\begin{table}[h]
  \centering
  \caption{CPU time and iteration numbers required to reach relative recovery
  error below $10^{-2}$, together with the final relative error.}
  \label{tab:time_to_target_error}
  \scriptsize
  \setlength{\tabcolsep}{3pt}
  \renewcommand{\arraystretch}{1.12}
  \resizebox{0.98\textwidth}{!}{%
  \begin{tabular}{l c c c c c c c c c c c c}
    \toprule
    \multirow{2}{*}{Algorithm}
    & \multicolumn{3}{c}{$\sigma=0$}
    & \multicolumn{3}{c}{$\sigma=0.001$}
    & \multicolumn{3}{c}{$\sigma=0.0015$}
    & \multicolumn{3}{c}{$\sigma=0.002$} \\
    \cmidrule(lr){2-4}
    \cmidrule(lr){5-7}
    \cmidrule(lr){8-10}
    \cmidrule(lr){11-13}
    & Time & Iters & RelErr
    & Time & Iters & RelErr
    & Time & Iters & RelErr
    & Time & Iters & RelErr \\
    \midrule

    \multicolumn{13}{c}{\textbf{Gaussian Matrix}} \\
    \midrule
    L1-FISTA
      & 0.0327 & 466 & 1.49e-03
      & 0.0251 & 352 & 2.79e-03
      & 0.0178 & 353 & 4.36e-03
      & 0.0182 & 356 & 6.19e-03 \\
    \textbf{ITP-C}
      & 0.0531 & 24 & 6.11e-16
      & 0.0435 & 14 & 2.65e-03
      & 0.0448 & 15 & 4.18e-03
      & 0.0445 & 15 & 5.90e-03 \\

    \midrule
    \multicolumn{13}{c}{\textbf{PDCT Matrix}} \\
    \midrule
    L1-FISTA
      & 0.0146 & 423 & 1.99e-03
      & 0.0127 & 307 & 3.63e-03
      & 0.0145 & 308 & 5.37e-03
      & 0.0171 & 310 & 7.34e-03 \\
    \textbf{ITP-C}
      & 0.0549 & 25 & 5.86e-16
      & 0.1043 & 13 & 3.47e-03
      & 0.0508 & 19 & 5.17e-03
      & 0.0610 & 21 & 7.01e-03 \\
    \bottomrule
  \end{tabular}%
  }
\end{table}

Table~\ref{tab:time_to_target_error} lists only L1-FISTA and ITP-C because
the other tested prior-free baselines, including L1-Reweighted, L1-2 PLDCA,
L1-2 PGD, and L1-2 ITAC, do not reach the $10^{-2}$ relative-error threshold
within $K_{\max}=800$ iterations in these tests. 

The results show two complementary aspects of recovery efficiency. L1-FISTA
has inexpensive first-order iterations and can have smaller raw CPU time in
some cases, but it requires several hundred iterations to reach the target
accuracy. By contrast, ITP-C reaches the same accuracy level within only a few
tens of iterations. The final relative errors further show that ITP-C reaches
near machine precision in the noiseless cases and attains lower final
errors than L1-FISTA in the noisy cases. This behavior is consistent with the
proposed mechanism: once a reliable active set is identified, the restricted
least-squares pursuit step recomputes the active coefficients and rapidly
reduces the recovery error. 
Overall, ITP-C improves iteration efficiency and
final recovery accuracy in these tests, while its wall-clock advantage depends
on the cost of the restricted least-squares correction.

\begin{figure}[!t]
    \centering
    \begin{minipage}[t]{0.49\textwidth}
      \centering
      \includegraphics[width=\linewidth]{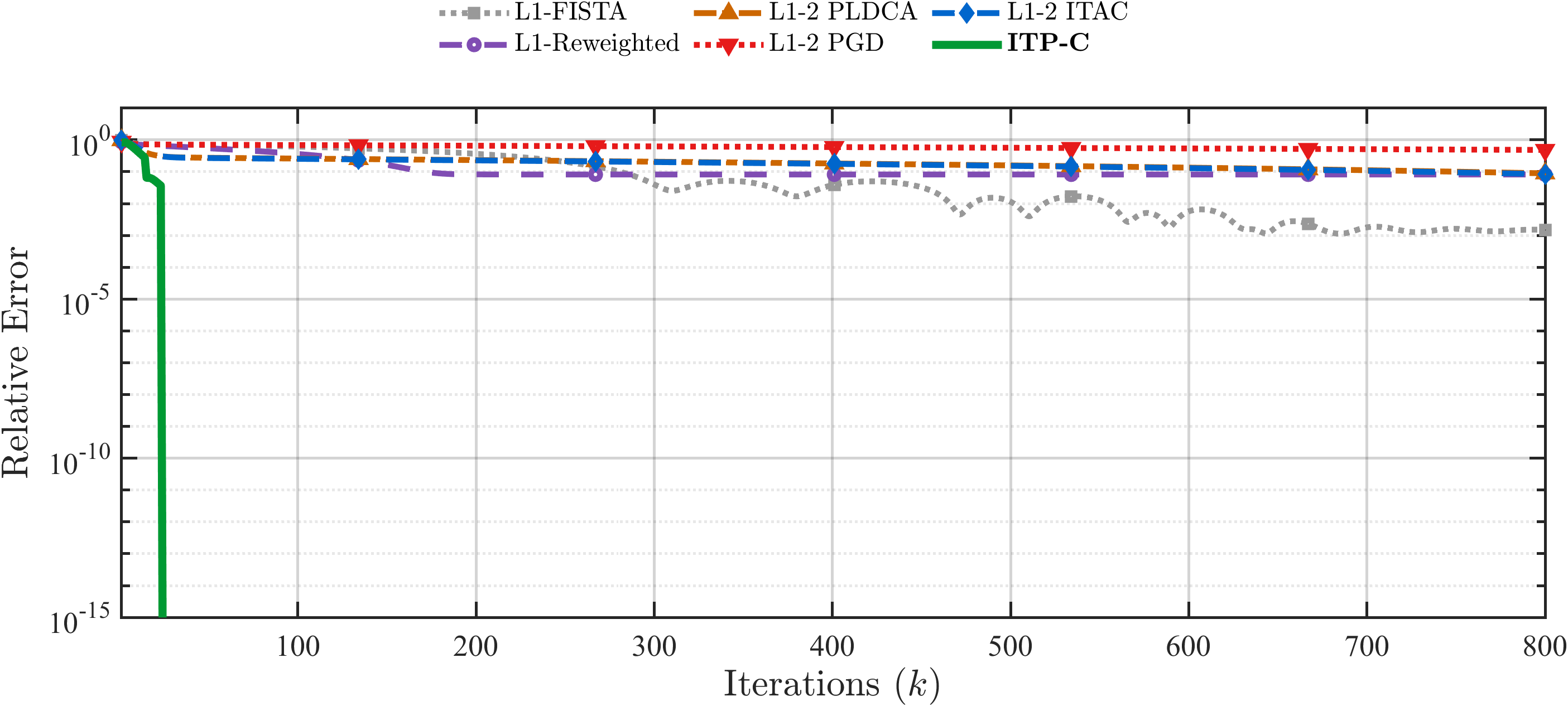}\\[-0.8mm]
      {\small (a) Noiseless scenario ($\sigma=0$)}
    \end{minipage}\hfill
    \begin{minipage}[t]{0.49\textwidth}
      \centering
      \includegraphics[width=\linewidth]{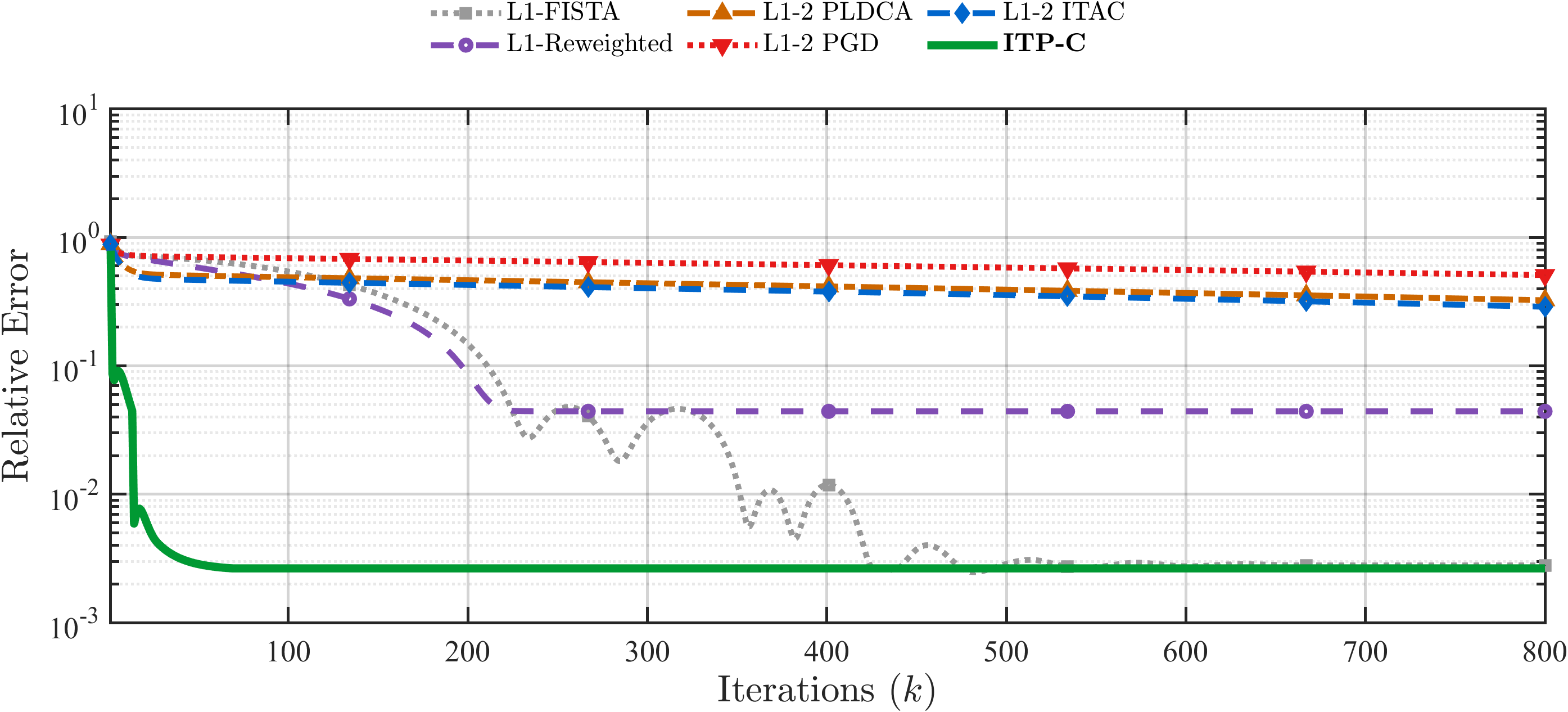}\\[-0.8mm]
      {\small (b) Noisy scenario ($\sigma=0.001$)}
    \end{minipage}
    \vspace{-0.08cm}
    \caption{Relative recovery error trajectories under the Gaussian sensing
    matrix. The two subfigures correspond to the noiseless case ($\sigma=0$)
    and the noisy case ($\sigma=0.001$).}
    \label{fig:convergence_gau}
\end{figure}

\begin{figure}[!t]
    \centering
    \begin{minipage}[t]{0.49\textwidth}
      \centering
      \includegraphics[width=\linewidth]{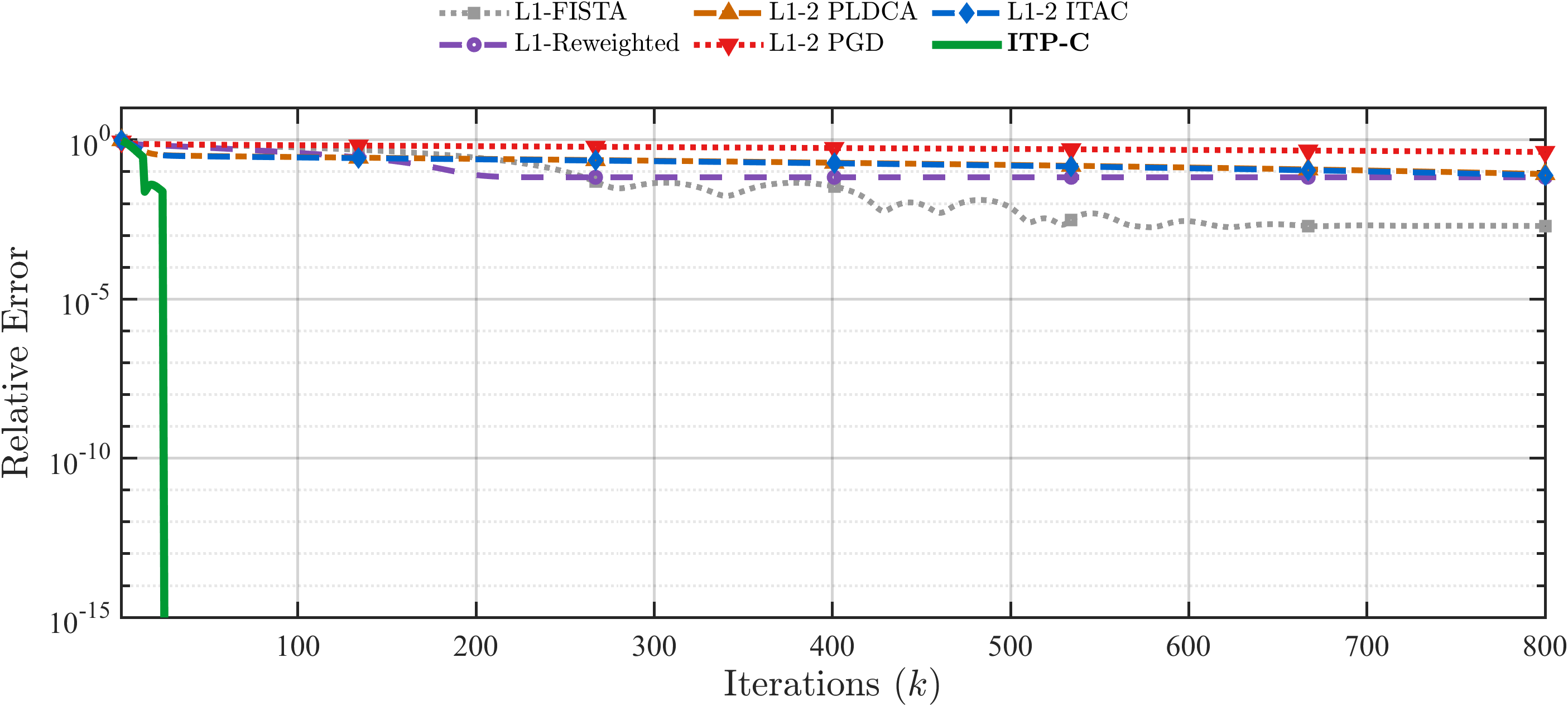}\\[-0.8mm]
      {\small (a) Noiseless scenario ($\sigma=0$)}
    \end{minipage}\hfill
    \begin{minipage}[t]{0.49\textwidth}
      \centering
      \includegraphics[width=\linewidth]{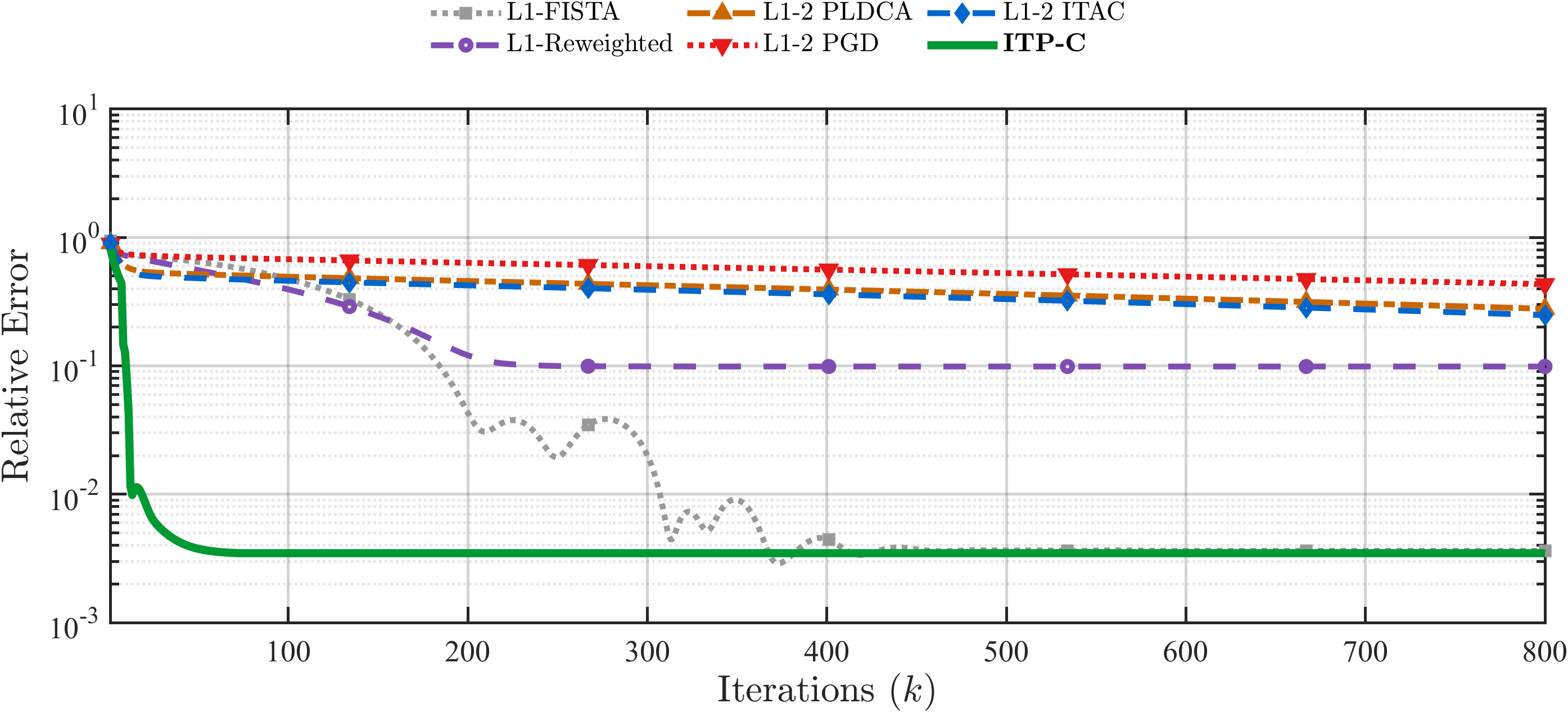}\\[-0.8mm]
      {\small (b) Noisy scenario ($\sigma=0.001$)}
    \end{minipage}
    \vspace{-0.08cm}
    \caption{Relative recovery error trajectories under the PDCT-type sensing
    matrix. The two subfigures correspond to the noiseless case ($\sigma=0$)
    and the noisy case ($\sigma=0.001$).}
    \label{fig:convergence_pdct}
\end{figure}


Figures~\ref{fig:convergence_gau} and~\ref{fig:convergence_pdct} show the
relative recovery error trajectories under the Gaussian and PDCT-type sensing
matrices. Several first-order or reweighted baselines, including L1-2 PGD,
L1-2 ITAC, L1-2 PLDCA, and L1-Reweighted, remain above the target accuracy
within the prescribed iteration budget. In contrast, L1-FISTA and ITP-C reach
the $10^{-2}$ relative-error level in these representative instances.

The two successful methods exhibit different error-reduction patterns. The
FISTA curve decreases gradually over many iterations, whereas the ITP-C curve
typically shows a sharp decrease once the pursuit step acts on a reliable
active set. In the noiseless case, this behavior can lead to nearly exact
recovery after support identification and least-squares refitting. In the noisy
case, the attainable error is limited by the measurement noise, but ITP-C still
enters the low-error regime within far fewer iterations. These observations are
consistent with the conditional oracle mechanism: after the active set becomes
sufficiently reliable, the restricted least-squares pursuit step reduces the
shrinkage left by thresholding-based updates. In this sense, the curves reflect the difference between gradual first-order
thresholding progress and a restricted second-order active-set correction on
the identified support.


We further assess the robustness of the compared methods by reporting empirical success
rates over repeated random trials. In
the phase-transition tests, the signal dimension is fixed at $n=512$. The left
panels fix the sparsity level at $s=25$ and vary the measurement ratio $m/n$,
whereas the right panels fix the number of measurements at $m=256$ and vary
the sparsity ratio $s/n$. For each grid point, $100$ independent trials are
performed with fixed random seeds. A trial is counted as successful if the
relative recovery error falls below $10^{-2}$ within $800$ iterations.

\begin{figure}[!htbp]
    \centering
    \begin{minipage}[t]{0.49\textwidth}
      \centering
      \includegraphics[width=\linewidth]{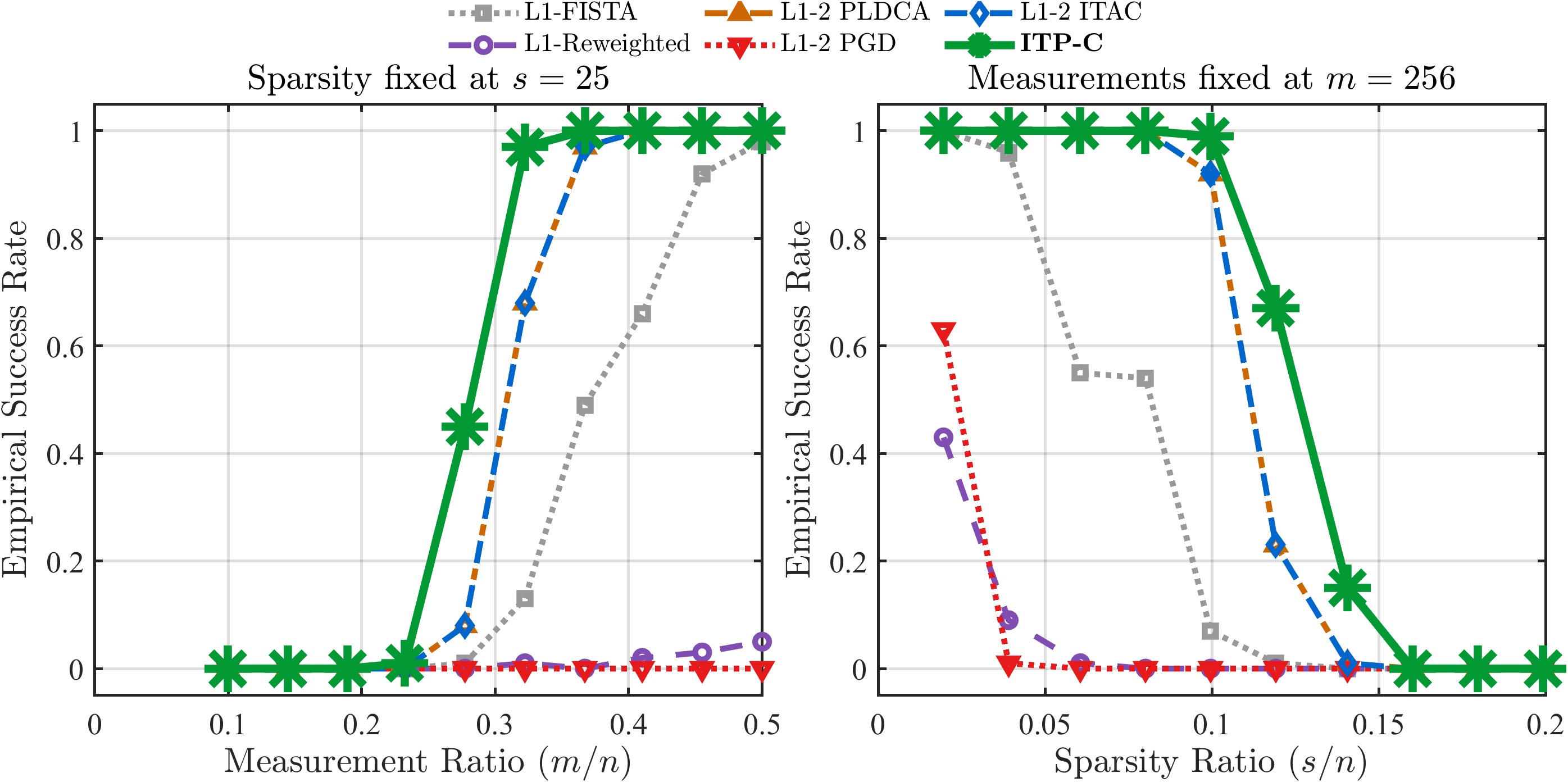}\\[-0.8mm]
      {\small (a) Noiseless scenario ($\sigma=0$)}
    \end{minipage}\hfill
    \begin{minipage}[t]{0.49\textwidth}
      \centering
      \includegraphics[width=\linewidth]{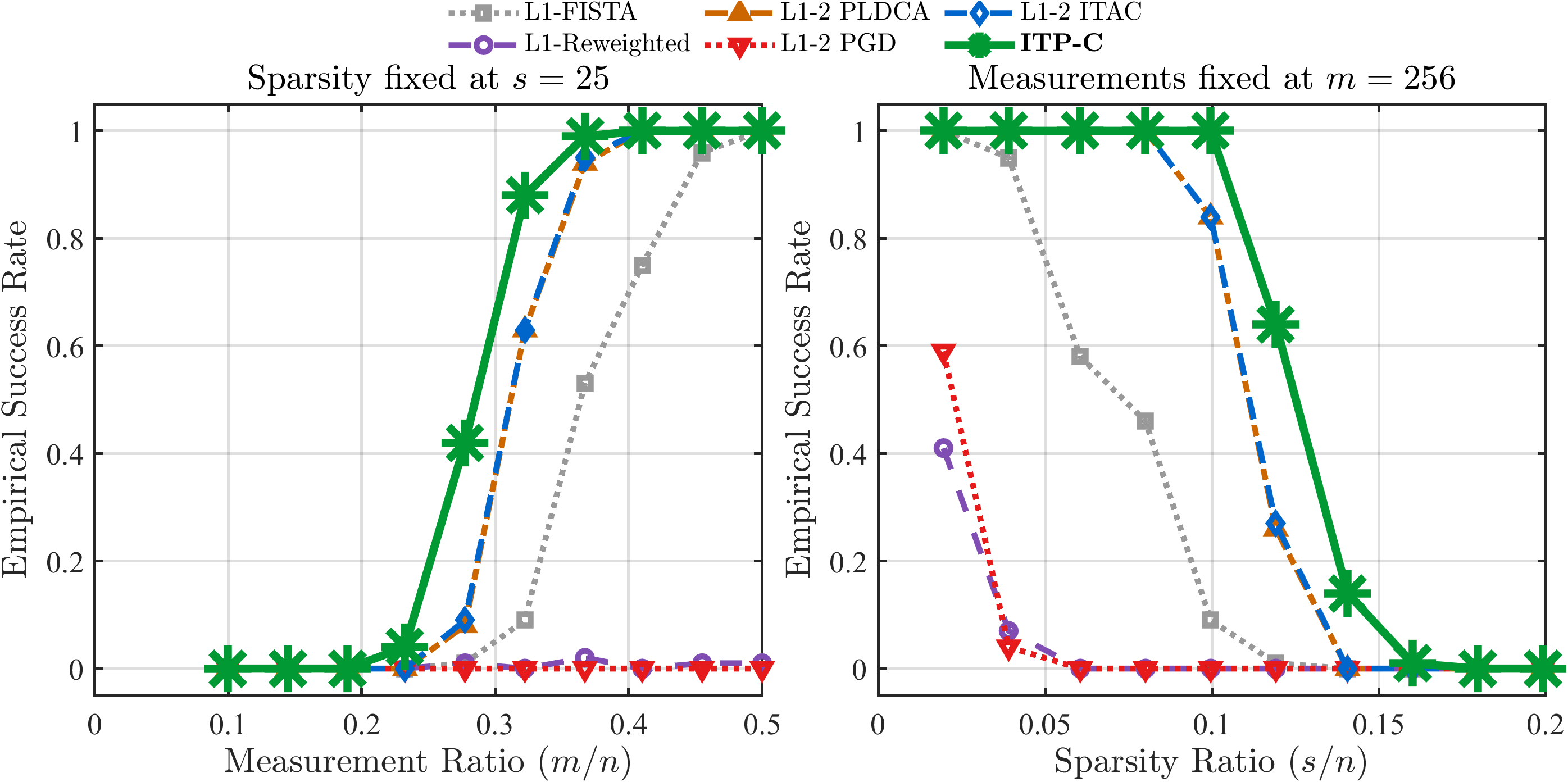}\\[-0.8mm]
      {\small (b) Noisy scenario ($\sigma=0.001$)}
    \end{minipage}
    \caption{Empirical success rate curves under the Gaussian sensing matrix.
    In each subfigure, the left panel reports the success rate versus the
    measurement ratio, and the right panel reports the success rate versus the
    sparsity ratio.}
    \label{fig:success_gau}
\end{figure}

\begin{figure}[!htbp]
    \centering
    \begin{minipage}[t]{0.49\textwidth}
      \centering
      \includegraphics[width=\linewidth]{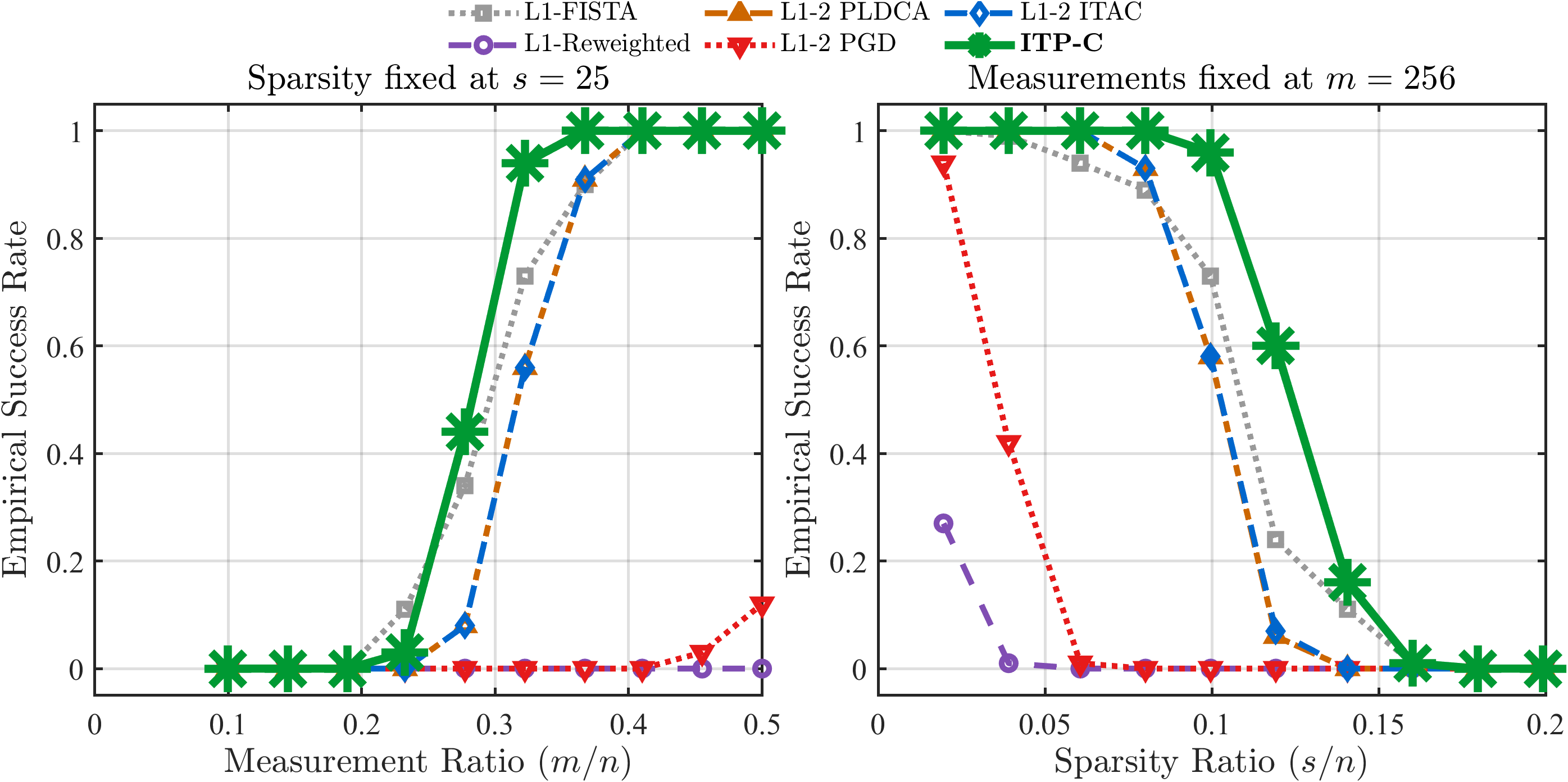}\\[-0.8mm]
      {\small (a) Noiseless scenario ($\sigma=0$)}
    \end{minipage}\hfill
    \begin{minipage}[t]{0.49\textwidth}
      \centering
      \includegraphics[width=\linewidth]{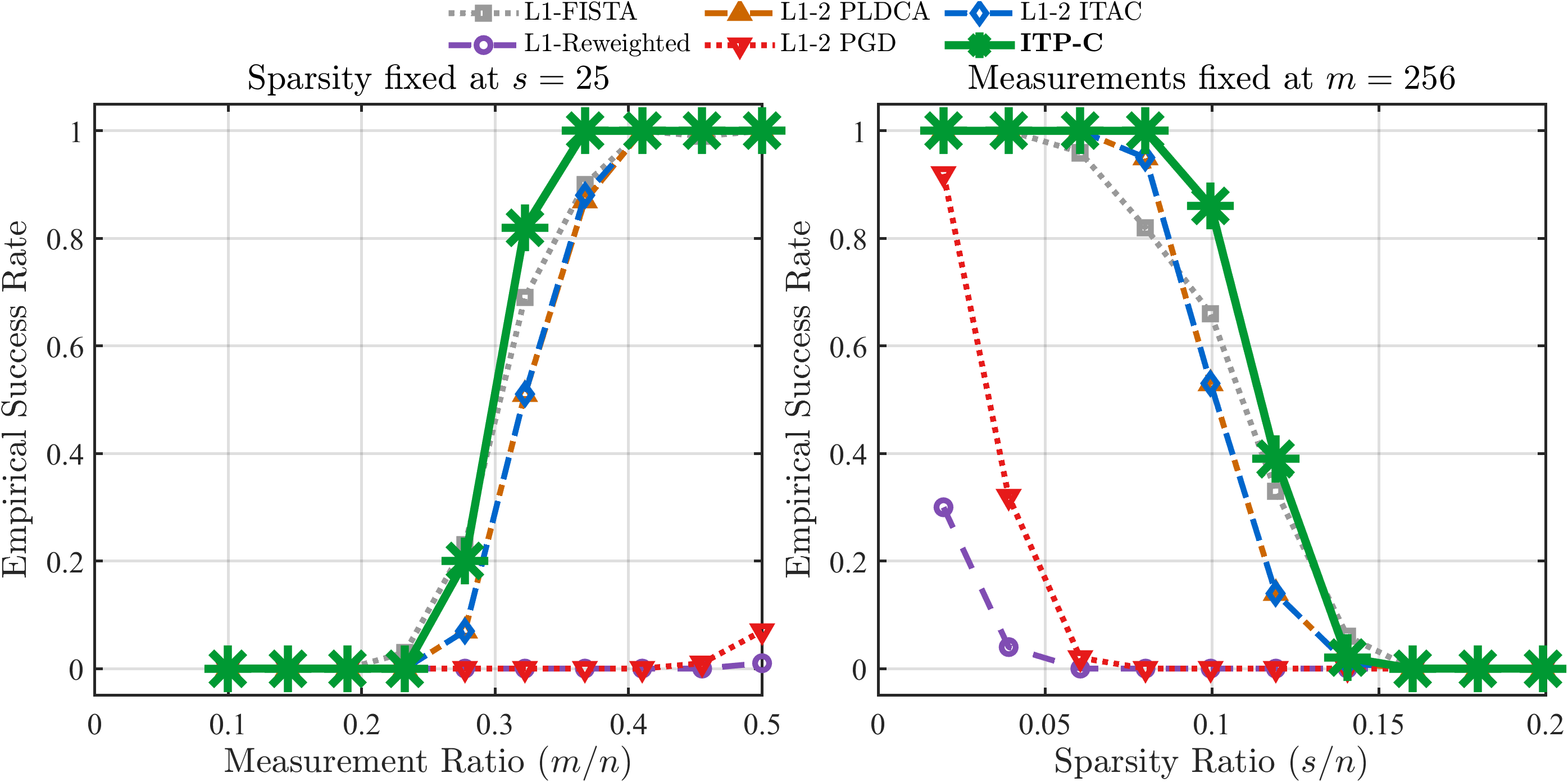}\\[-0.8mm]
      {\small (b) Noisy scenario ($\sigma=0.001$)}
    \end{minipage}
    \caption{Empirical success rate curves under the PDCT-type sensing matrix.
    In each subfigure, the left panel reports the success rate versus the
    measurement ratio, and the right panel reports the success rate versus the
    sparsity ratio.}
    \label{fig:success_pdct}
\end{figure}

Figures~\ref{fig:success_gau} and~\ref{fig:success_pdct} compare the success
rates over repeated random trials. Under both sensing structures, ITP-C
consistently achieves the highest success rates among the compared methods,
especially at lower measurement ratios and larger sparsity levels. The
improvement remains visible in the noisy tests, showing that the proposed
descent-checked pursuit strategy is robust across a range of sampling and
noise regimes.


\FloatBarrier

\subsection{Application to image reconstruction}
\label{subsec:image}

We finally evaluate ITP-C on image reconstruction examples. Natural images are
not exactly sparse in the pixel domain, but they often admit compressible
representations under wavelet transforms. In the experiments, each image is
represented using a four-level Symlets-8 wavelet transform and measured by
Gaussian random projections. Gaussian measurement noise with standard deviation
$\sigma=10^{-3}$ is added in all tests. Three benchmark images, Peppers,
Circuit, and Coins, are considered under sampling rates
$\delta\in\{0.45,0.50,0.55\}$.

For an image $u\in[0,1]^{N_1\times N_2}$ and its reconstruction
$\hat u$, the PSNR is computed as
\[
{\rm PSNR}(\hat u,u)
=
10\log_{10}
\left(
\frac{N_1N_2}{\|\hat u-u\|_F^2}
\right).
\] 

\begin{table}[htbp]
  \centering
  \caption{Comparison of PSNR (dB) for different algorithms and sampling rates ($\delta$) in image reconstruction with Gaussian measurement noise $\sigma=10^{-3}$.}
  \label{tab:psnr_comparison}
  \small
  \setlength{\tabcolsep}{3pt}
  \renewcommand{\arraystretch}{1.08}
  \resizebox{0.98\textwidth}{!}{%
  \begin{tabular}{l c c c c c c c}
    \toprule
    Image name & $\delta$ & L1-FISTA & L1-Reweighted & L1-2 PLDCA & L1-2 PGD & L1-2 ITAC & \textbf{ITP-C} \\
    \midrule
    \multirow{3}{*}{Peppers}
      & 0.45 & 17.62 & 12.72 & 17.25 & 12.56 & 17.36 & \textbf{23.11} \\
      & 0.50 & 20.33 & 13.18 & 18.75 & 12.97 & 18.88 & \textbf{27.06} \\
      & 0.55 & 26.88 & 14.01 & 21.62 & 13.75 & 21.73 & \textbf{27.88} \\
    \cmidrule(lr){1-8}
    \multirow{3}{*}{Circuit}
      & 0.45 & 17.02 & 13.06 & 16.27 & 12.84 & 16.34 & \textbf{23.66} \\
      & 0.50 & 20.71 & 13.63 & 19.69 & 13.37 & 19.79 & \textbf{25.75} \\
      & 0.55 & 25.66 & 14.22 & 21.03 & 13.90 & 21.18 & \textbf{28.74} \\
    \cmidrule(lr){1-8}
    \multirow{3}{*}{Coins}
      & 0.45 & 14.41 & 10.72 & 14.89 & 10.61 & 14.97 & \textbf{20.41} \\
      & 0.50 & 19.25 & 11.79 & 18.29 & 11.64 & 18.42 & \textbf{30.07} \\
      & 0.55 & 19.94 & 11.52 & 18.29 & 11.37 & 18.47 & \textbf{29.90} \\
    \bottomrule
  \end{tabular}
  }
\end{table}

Table~\ref{tab:psnr_comparison} reports the PSNR values of the compared
methods. ITP-C achieves the highest PSNR in all reported cases, showing that
the proposed descent-checked pursuit strategy improves reconstruction quality.



Figure~\ref{fig:visual_comp} provides a visual comparison at sampling rate
$\delta=0.50$ under Gaussian measurement noise with $\sigma=10^{-3}$. The
visual results are consistent with the PSNR values in
Table~\ref{tab:psnr_comparison}. Compared with the other tested methods, ITP-C
produces fewer visible artifacts and better preserves fine image structures.

\begin{figure}[!tbp]
    \centering
    \includegraphics[width=0.96\textwidth]{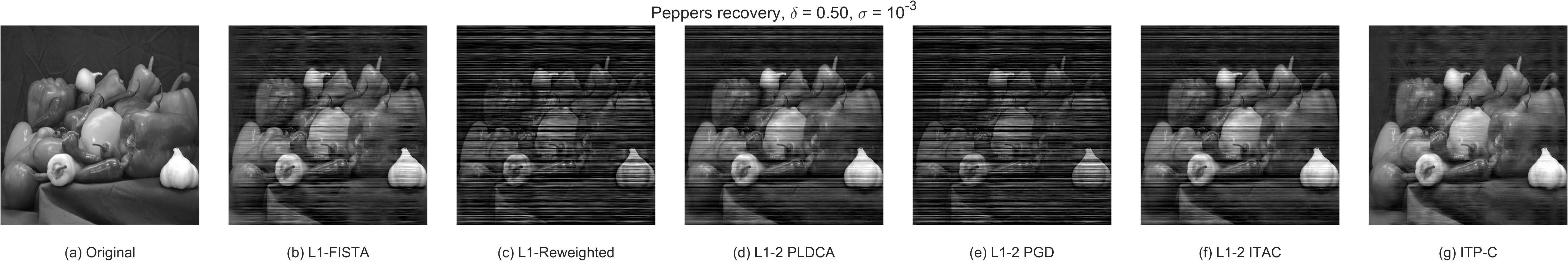}\\[-0.8mm]
    {\small (a) Peppers recovery ($\delta=0.50$, $\sigma=10^{-3}$)}\\[0.3mm]
    \includegraphics[width=0.96\textwidth]{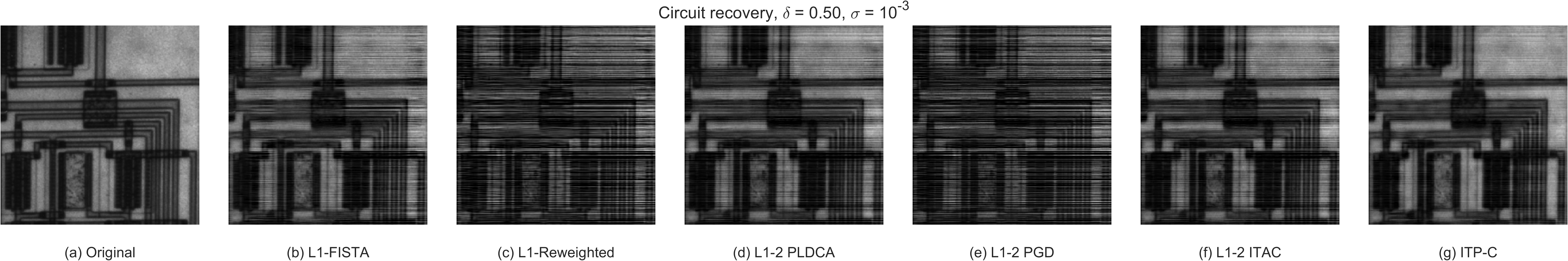}\\[-0.8mm]
    {\small (b) Circuit recovery ($\delta=0.50$, $\sigma=10^{-3}$)}\\[0.3mm]
    \includegraphics[width=0.96\textwidth]{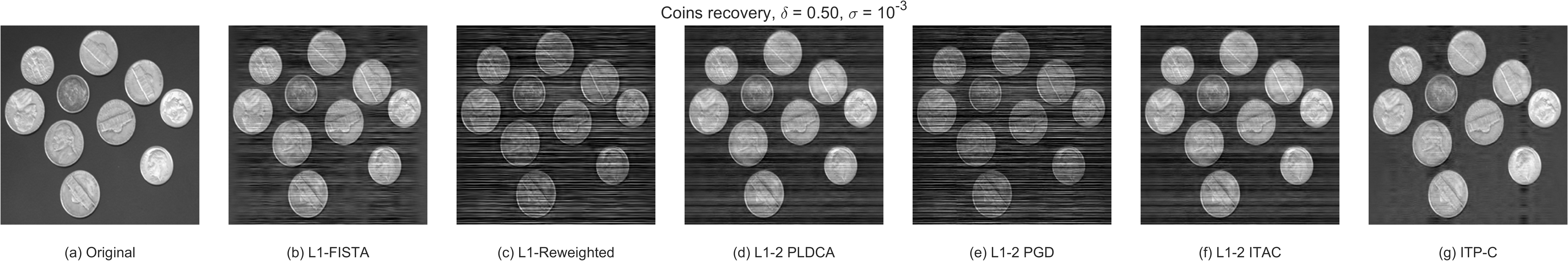}\\[-0.8mm]
    {\small (c) Coins recovery ($\delta=0.50$, $\sigma=10^{-3}$)}
    \caption{Visual comparison of reconstructed benchmark images at a sampling
    rate of $\delta=0.50$ under Gaussian measurement noise $\sigma=10^{-3}$.
    Within each row, the first panel displays the original ground-truth image,
    and the remaining panels show the recovery results of the compared
    algorithms.}
    \label{fig:visual_comp}
\end{figure}

\section{Conclusion}
\label{sec:conclusion}

In this paper, we proposed an iterative thresholding pursuit method with
continuation for $\ell_{1-2}$-regularized sparse recovery. The method combines
the adaptive support-identification ability of the $\ell_{1-2}$ proximal
thresholding step with a restricted least-squares pursuit correction on the
identified support. A strict descent check is incorporated to control the
possible instability of the pursuit correction and to maintain the descent
property of the continuation scheme. The resulting algorithm does not require
the true sparsity level as an input and provides a prior-free active-set
refitting mechanism.

We established the convergence properties of the proposed method. In
particular, the generated sequence is shown to be bounded, to satisfy a
sufficient decrease property, and to converge to a critical point of the target
objective under the Kurdyka--\L{}ojasiewicz framework. We also proved a local
support identification result and a conditional oracle property: once the
correct support is identified and the pursuit candidate is accepted, the method
coincides with the oracle restricted least-squares estimator and satisfies a
stable noise-dependent error bound.

Numerical experiments on synthetic sparse recovery and image reconstruction
demonstrated the effectiveness of the descent-checked pursuit strategy. Future work will investigate efficient implementations of the restricted
least-squares pursuit step, including preconditioning and warm-started
iterative solvers, as well as extensions to structured sparsity models and
broader sensing structures.



\end{document}